\documentclass[11pt]{article}
\usepackage[utf8]{inputenc}
\usepackage[margin=1in]{geometry}
\usepackage{amssymb,amsmath,amsthm,graphicx,caption}
\usepackage{booktabs}
\usepackage{longtable}
\usepackage{enumitem}
\usepackage{rotating}
\usepackage{pdflscape}
\usepackage[hidelinks]{hyperref}
\usepackage{url}
\usepackage{cite}

\newtheorem{theorem}{Theorem}
\newtheorem{proposition}[theorem]{Proposition}
\newtheorem{lemma}[theorem]{Lemma}
\newtheorem{corollary}[theorem]{Corollary}

\newtheorem{remark}{Remark}

\graphicspath{{figures/}}

\title{iSTAR: an algebraic-collapse framework for variational reduction in quantum-inspired continuous Ising solvers}

\author{Bowen Liu\thanks{eBay Inc.\ Shanghai, China.
\href{mailto:boweliu@ebay.com}{\texttt{boweliu@ebay.com}}}
\and
Dongmei Xiao\thanks{Corresponding author.
School of Mathematical Sciences,
Shanghai Jiao Tong University, Shanghai 200240, China.
\href{mailto:xiaodm@sjtu.edu.cn}{\texttt{xiaodm@sjtu.edu.cn}}}}

\date{}

\begin{document}

\maketitle

\begin{abstract}
Continuous Ising solvers embed a discrete optimization problem into a
continuous dynamical system and recover the spin configuration by sign
readout, but dense interaction evaluation gives an
$\mathcal{O}(N^2)$-per-step cost. We show that this cost is not
intrinsic: during late-stage simulated bifurcation the trajectory
collapses onto a lower-dimensional active subspace, and saturated
coordinates can be eliminated exactly by a variational frozen-set
identity whose couplings fold into an induced field on the
unresolved subsystem. We prove large-parameter recovery for the
external-field quartic model, the hard-box limit of ballistic
confinement, and a robust-margin freezing criterion. The resulting
algorithm, iSTAR (Ising Stable-set Tail-Aware Reduction), exploits
this collapse by detecting stabilized coordinates and continuing only
on the active tail. An online certified implementation on the
G-set benchmark preserves the same-seed baseline in all
runs and removes on average 64.4\% of the dense interaction work.
\end{abstract}

\medskip
\noindent\textbf{Keywords:} Ising model, simulated bifurcation, quantum-inspired optimization, variational reduction, frozen-set method, active-set reduction, combinatorial optimization

\section{Introduction}

Many optimization problems in science and engineering reduce to the
search for an optimal binary configuration. Ising formulations
provide a common language for Max-Cut, graph partition, community
detection, circuit layout, and other NP-hard
problems~\cite{ref1,ref2,ref3,ref4,ref5,ref6,ref7,ref8,ref9,ref10,ref11,ref12,ref13,ref14,ref15,ref16,ref17,ref18,ref19,ref20,ref21,ref22,ref23,ref24,ref25,ref26}.
This broad applicability has motivated a wide range of physics-inspired
and quantum-inspired solvers in which a discrete objective is embedded
into a continuous dynamical system and the final binary state is
recovered by sign
readout~\cite{ref27,ref28,ref29,ref30,ref31,ref32,ref33,ref34,ref35,ref36,ref37,ref38,ref39,ref40,ref41}.
Recent work has also improved simulated-bifurcation-type solvers by
modifying their dynamics or adding auxiliary
fluctuations~\cite{ref20}. Recent applications further indicate that
such Ising-based methods can be effective in molecular docking,
molecular conformation generation, and related
tasks~\cite{ref16,ref17,ref18,ref19}.

The mathematical question behind these methods is natural. A
simulated bifurcation solver evolves in a continuous state space,
whereas the target problem is discrete. One therefore asks when the
minimizers, or at least the natural sign readout, of the continuous
model remain faithful to the underlying Ising objective. This is first
a variational question about the structure of the landscape, before it
becomes a question about any particular numerical trajectory.

Two distinct variational pictures arise in the simulated bifurcation
literature. One is the quartic soft-spin landscape underlying
adiabatic simulated bifurcation (aSB). The other is the
hard-confinement, or saturation, picture underlying ballistic
simulated bifurcation (bSB). These two pictures are related, but they
lead to different mathematical questions and different reduction
mechanisms. Related continuous relaxations have also appeared in other
recent work, including the Free-Energy Machine of Shen et
al.~\cite{ref34}; here, however, our focus is specifically on the
simulated bifurcation family and on the reduction mechanism induced by
late-stage stabilization.

The central question of this paper is whether continuous Ising
dynamics contain an intrinsic active-set structure. Our answer is
variational. We show that once a subset of spin signs is fixed, those
coordinates need not be treated as merely inactive numerical
variables: they can be eliminated exactly. Their couplings become an
induced external field on the unresolved subsystem, and the original
Ising instance reduces to a lower-dimensional Ising problem on the
remaining coordinates. This frozen-set reduction principle is the
conceptual core of the paper (Fig.~\ref{fig:schematic}).

\begin{figure}[t]
\centering
\includegraphics[width=0.96\textwidth]{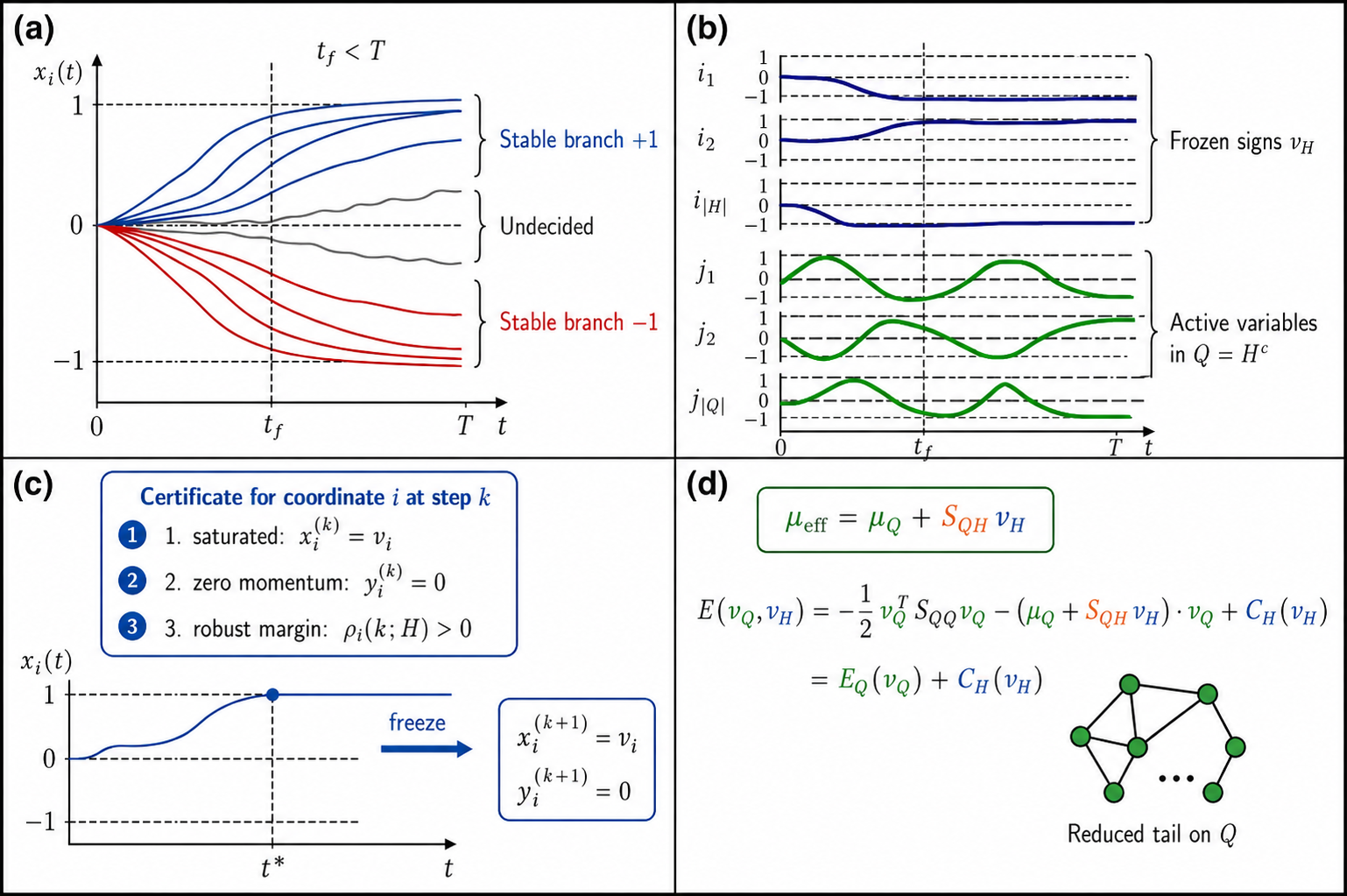}
\caption{Illustration of iSTAR. (a) In simulated bifurcation, many
coordinates select stable branches before the nominal terminal time.
(b) The trajectory therefore induces a certified split into frozen
coordinates $H$ and active coordinates $Q = H^c$. (c) Coordinates are
frozen only after the robust branch-freezing certificate is satisfied;
the compact margin $\rho_i(k;H)$ used in the panel is defined in
Section~\ref{sec:theory}. (d) The frozen signs $\mathbf{v}_H$ are
eliminated algebraically; their couplings induce the effective field
$\mu_{\mathrm{eff}} = \mu_Q + S_{QH}\mathbf{v}_H$, and the tail
computation continues on the reduced subsystem $Q$.}
\label{fig:schematic}
\end{figure}

We consider the Ising problem with external field
\begin{equation}
\min_{\mathbf{v}\in\{-1,1\}^n} E(\mathbf{v}) := -\tfrac{1}{2} \mathbf{v}^{\mathsf{T}} S \mathbf{v} - \mu \cdot \mathbf{v},
\label{eq:ising}
\end{equation}
where $S = (s_{ij})_{n \times n}$ is a symmetric interaction matrix
with $s_{ii} = 0$ and $\mu \in \mathbb{R}^n$ is the external field.
This discrete objective is the common target of all the simulated
bifurcation models studied in this paper. We write
$\mathcal{C} := \{-1,1\}^n$ for the corresponding set of spin
configurations.

Simulated bifurcation was introduced through two closely related
Hamiltonian models~\cite{ref21,ref22}. Both evolve a pair of conjugate
variables $(\mathbf{x}, \mathbf{y}) \in \mathbb{R}^{2n}$ under a
time-dependent Hamiltonian whose potential part embeds the Ising
couplings.

\medskip
\noindent\textbf{aSB (adiabatic).}
The aSB Hamiltonian includes a quartic double-well term that creates
the soft-spin bifurcation landscape:
\begin{equation}
H_{\mathrm{aSB}}(\mathbf{x}, \mathbf{y}, t)
= \sum_{i=1}^{n} \frac{y_i^2}{2}
+ \sum_{i=1}^{n} \frac{x_i^4}{4}
+ \frac{\beta - \alpha(t)^2}{2} \|\mathbf{x}\|^2
- \frac{1}{2} \mathbf{x}^{\mathsf{T}} S \mathbf{x}
- \alpha(t)\, \mu \cdot \mathbf{x},
\label{eq:Ham_aSB}
\end{equation}
where $\alpha(t) \ge 0$ is a non-decreasing bifurcation parameter and
$\beta > 0$ is fixed. The potential part of $H_{\mathrm{aSB}}$ is the
quartic landscape $U_{\mathrm{aSB}}$ analyzed in
Section~\ref{sec:theory}. Hamilton's equations give
\begin{equation}
\frac{dx_i}{dt} = \frac{\partial H_{\mathrm{aSB}}}{\partial y_i} = y_i,
\qquad
\frac{dy_i}{dt} = -\frac{\partial H_{\mathrm{aSB}}}{\partial x_i}
= -\Bigl( x_i^3 + (\beta-\alpha^2)x_i
- \sum_{j=1}^{n} s_{ij} x_j - \alpha \mu_i \Bigr).
\label{eq:aSB_EOM}
\end{equation}

\noindent\textbf{bSB (ballistic).}
The bSB model~\cite{ref22} replaces the smooth quartic confinement of
aSB by perfectly inelastic walls at $x_i = \pm 1$.  Following the
notation of~\cite{ref22}, the potential is
\begin{equation}
V_{\mathrm{bSB}}(\mathbf{x})
= \frac{a_0 - a(t)}{2} \sum_{i=1}^{n} x_i^2
- \frac{c_0}{2} \sum_{i=1}^{n}\sum_{j=1}^{n} J_{ij}\, x_i x_j,
\qquad |x_i| \le 1 \;\; \forall i,
\label{eq:V_bSB}
\end{equation}
and $V_{\mathrm{bSB}} = \infty$ otherwise, where $a_0, c_0 > 0$ are
constants, $a(t)$ increases from $0$ to $a_0$, and $J$ is the
symmetric coupling matrix with zero diagonal.  The Hamiltonian is
\begin{equation}
H_{\mathrm{bSB}}(\mathbf{x}, \mathbf{y}, t)
= \frac{a_0}{2} \sum_{i=1}^{n} y_i^2
+ V_{\mathrm{bSB}}(\mathbf{x}),
\label{eq:Ham_bSB}
\end{equation}
In the present work we additionally include an external-field term
$-\mu \cdot \mathbf{x}$ in the potential, specialize to
$a_0 = c_0 = 1$, and identify $\alpha(t)^2 = a(t)$ and
$S = J$.  The discrete Euler update used in iSTAR is then given
explicitly in the Appendix (Proposition~\ref{prop:one_step_app}).

The paper develops the frozen-set reduction idea along two
complementary theoretical lines. For the aSB-type quartic potential
with external field, we prove a large-$\alpha$ recovery theorem
showing that, under explicit quantitative conditions, every global
minimizer of the continuous problem decodes by sign readout to a
minimizer of the Ising problem. For the bSB-type hard-confinement
model, we study a soft-to-hard approximation scheme, prove that the
limiting hard-box problem reduces to the same discrete objective under
a positive-semidefinite condition, and derive explicit frozen-set and
persistent-freezing criteria for late-stage reduction. These results
lead to iSTAR, which implements the frozen-set principle by detecting
stable coordinates, transferring their couplings to the active
subsystem, and continuing only the unresolved tail.

\paragraph{Our contribution.}
The paper has three main contributions.
\begin{enumerate}[leftmargin=*]
\item For the aSB-type quartic potential, we prove a large-$\alpha$
recovery theorem: above an explicit threshold depending only on $S$,
$\mu$, and $\beta$, every global minimizer of $U_{\mathrm{aSB}}$ has
sign vector equal to an Ising minimizer. This gives a direct
continuous-to-discrete consistency theorem for the external-field aSB
landscape.

\item For the bSB-type confinement family, we analyze the soft
objectives $U_{p,\alpha}$ and their hard-box limit. We prove
epi-convergence of the confinement term, identify the limiting
hard-box variational problem, show that under a positive-semidefinite
condition the limiting minimization problem reduces to the Ising
problem on the vertices, and derive frozen-set reduction and
persistent-freezing principles for the late-stage regime.

\item Guided by these reduction principles, we propose iSTAR (Ising
Stable-set Tail-Aware Reduction), an active-set reduction framework
for SB-type continuous Ising solvers. Its primary certified
implementation is tested on bSB, where coordinates satisfying the
robust freezing condition are progressively removed from the active
dynamics and folded into the induced field. On the Goto et
al.~\cite{ref22} \href{https://www.science.org/doi/suppl/10.1126/sciadv.abe7953}{Table~S2} G1--G54 protocol with nonzero external
fields, this online certified version triggers in all $540$
seed--instance runs, matches the same-seed full bSB baseline in all
runs, and removes on average $56.32\%$ of the dense interaction work.
Fixed-grid probe sweeps are used as diagnostic evidence for the
breadth of late-tail reducibility, while supporting aSB/dSB
experiments indicate that the same principle is not tied to one
simulated bifurcation variant.
\end{enumerate}

Section~\ref{sec:theory} develops the variational recovery theory,
including the large-$\alpha$ aSB recovery theorem, the bSB
hard-confinement reduction theorem, the frozen-set reduction principle,
and the finite-time freezing guarantee. Section~\ref{sec:methods}
describes the solvers, the iSTAR algorithm, and the benchmark protocol,
building on the theoretical foundations of Section~\ref{sec:theory}.
Section~\ref{sec:results} reports the experimental evidence.
Section~\ref{sec:discussion} discusses interpretation, limitations,
and next steps.

\section{Theory}
\label{sec:theory}

\subsection{Variational recovery principles}

We develop the variational mechanism behind the reduction. The
physical picture is simple. Simulated bifurcation starts from
continuous oscillator amplitudes, but the useful information at late
times is mostly discrete: each coordinate has chosen one of two sign
branches. In the aSB model, this branch selection is created by a
smooth quartic double-well landscape whose wells move outward as the
bifurcation parameter grows. In the bSB model, the same sign selection
is enforced by saturation at the box boundary. In both cases, once the
amplitude degree of freedom has stabilized, the remaining question is
which sign pattern is selected; the variational statements below
explain when that sign pattern is ordered by the original Ising
energy.

\subsubsection{aSB landscape recovery.}

The adiabatic simulated bifurcation model gives a smooth quartic
relaxation of the Ising problem. In the normalized variables used
here, the associated external-field potential is
\begin{equation}
U_{\mathrm{aSB}}(\mathbf{x}) =
\sum_{i=1}^{n} \tfrac{1}{4} x_i^4
+ \tfrac{\beta - \alpha^2}{2} \mathbf{x}^{\mathsf{T}} \mathbf{x}
- \tfrac{1}{2} \mathbf{x}^{\mathsf{T}} S \mathbf{x}
- \alpha \mu \cdot \mathbf{x},
\quad \mathbf{x} \in \mathbb{R}^n,
\label{eq:pote}
\end{equation}
where $\alpha$ is the bifurcation parameter and $\beta > 0$ is fixed.
The zero-field recovery theory was developed
in~\cite{ref15}; the external-field term is the part needed for the
reduction framework in this paper.

The physical role of the pumping parameter is to separate the two
branches of each oscillator. When $\alpha$ is small, the origin is
energetically relevant and the continuous state does not yet carry a
reliable spin readout. As $\alpha$ grows, the quartic potential
creates two deep wells in each coordinate, and the coupling terms tilt
the product of these wells so that different sign patterns have
different depths. In the large-$\alpha$ regime, the wells of
$U_{\mathrm{aSB}}$ concentrate near the lifted spin states
$\alpha \mathbf{v}$, $\mathbf{v} \in \{-1,1\}^n$. At those lifted
states the continuous and Ising objectives share the same ordering, a
fact that follows from a direct algebraic computation.

\begin{lemma}
\label{lem:lifted}
For every spin configuration $\mathbf{v} \in \{-1,1\}^n$,
\begin{equation}
U_{\mathrm{aSB}}(\alpha \mathbf{v}) = C_{\alpha,\beta,n} + \alpha^2 E(\mathbf{v}),
\label{eq:lift-order}
\end{equation}
where
$C_{\alpha,\beta,n} = \tfrac{n}{4}\alpha^4 + \tfrac{n}{2}(\beta-\alpha^2)\alpha^2$
is independent of $\mathbf{v}$. Consequently, the ordering of the
continuous potential at the lifted spin states coincides with the
ordering of the Ising energy, amplified by $\alpha^2$.
\end{lemma}

\noindent Thus, at the level of branch centers, aSB is a
large-amplitude embedding of the same Ising energy. The only remaining
issue is to rule out global minimizers away from the lifted spin wells
and to control the error between each well minimizer and its center.

We state the recovery theorem with the explicit threshold used in the
proof. Let $R_* = \max_i \sum_j |s_{ij}|$ be the maximum row-sum
coupling scale and $M_* = \max_i |\mu_i|$ the external-field scale.
Let $\Delta_*$ be the smallest nonzero separation between two Ising
energies, with $\Delta_* = 1$ if the Ising energy is constant on
$\mathcal{C}$, and define
$C_* = n(5/4 + 3|\beta|/2 + 3 R_*/2 + M_*)$. We use the explicit
threshold
\begin{equation}
\alpha_* := \max\left\{
1,\; C_*,\; \frac{2 C_*}{\Delta_*},\;
5(2 R_* + 2|\beta| + M_* + 1),\;
\sqrt{R_* + |\beta| + \tfrac{M_*}{2}},\;
\sqrt{3 + |\beta| + R_*}
\right\}.
\label{eq:astar}
\end{equation}

\begin{theorem}
\label{thm:asb_main}
For fixed $S$, $\mu$, and $\beta$, let $U_{\mathrm{aSB}}$ be the
external-field quartic potential in \eqref{eq:pote}. If
$\alpha > \alpha_*$ and $\mathbf{x}_0$ is a global minimizer of
$U_{\mathrm{aSB}}$, then $\operatorname{sgn}(\mathbf{x}_0)$ is a global
minimizer of the Ising problem \eqref{eq:ising}.
\end{theorem}

The proof is given in Appendix~\ref{sec:app_a1} and has three steps.
First, a barrier estimate shows that each lifted spin state has a
nearby interior well whose value differs from the lifted value only by
$O(\alpha)$. Second, the critical-point equation and Hessian test
exclude local minima on the central branch near the origin, so local
minima must lie near $\pm \alpha$ in each coordinate. Third, the
lifted-state identity \eqref{eq:lift-order} amplifies Ising energy
gaps by $\alpha^2$, which dominates the $O(\alpha)$ well error for
$\alpha > \alpha_*$. Hence a continuous global minimizer cannot decode
to a nonoptimal spin vector. Here $R_*$ controls the coupling scale,
$M_*$ controls the external-field scale, $\Delta_*$ is the smallest
nonzero Ising energy separation, and $C_*$ collects the uniform
well-error constants used in the proof. The threshold $\alpha_*$ is
intentionally explicit rather than sharp; its role is to make the
recovery statement quantitative.

This result is a static variational recovery theorem rather than a
finite-time trajectory theorem. Its role in the present paper is to
justify the continuous-to-discrete alignment behind aSB-type
reductions with external fields. The late-stage algorithmic reduction
studied below uses the same principle in a different form: stable
coordinates are treated as fixed spins, and their influence becomes an
induced external field on the unresolved subsystem.

\subsubsection{bSB hard-confinement reduction.}

The ballistic version of simulated bifurcation drops the quartic
term of aSB and instead enforces $|x_i| \le 1$ by a hard saturation
rule~\cite{ref22}. At the
variational level, this leads to a hard-confinement model: the
quadratic drift is retained inside the box, while states outside the
box are excluded. We write
\begin{equation}
U_{\Box,\alpha}(\mathbf{x}) :=
\tfrac{1 - \alpha^2}{2} \|\mathbf{x}\|^2
- \tfrac{1}{2} \mathbf{x}^{\mathsf{T}} S \mathbf{x}
- \mu \cdot \mathbf{x},
\qquad \mathbf{x} \in [-1,1]^n.
\label{eq:hardbox}
\end{equation}
On the open box,
$\nabla U_{\Box,\alpha}(\mathbf{x}) = (1 - \alpha^2)\mathbf{x} - S\mathbf{x} - \mu$,
which is the static interior field associated with the normalized bSB
confinement.

The physical interpretation is that bSB removes the need to resolve
large amplitudes. Once a coordinate reaches the wall, the clipping
rule prevents it from drifting farther outward and the momentum reset
removes the outward motion. The relevant amplitude is therefore fixed
at $\pm 1$, and the remaining optimization problem is a sign-selection
problem on the faces and vertices of the box. The hard-box objective
is the static variational object that retains the interior force field
while replacing the dynamical saturation rule by the constraint
$\mathbf{x} \in [-1,1]^n$.

To connect this constrained problem with smooth soft-spin landscapes,
we use the soft wall penalty
$\Phi_p(\mathbf{x}) := \sum_{i=1}^n |x_i|^{2p}/(2p)$, with $p \ge 2$.
The corresponding soft-confinement objective is
$U_{p,\alpha}(\mathbf{x}) := \Phi_p(\mathbf{x}) + U_{\Box,\alpha}(\mathbf{x})$.

\begin{proposition}
\label{prop:epi}
As $p \to \infty$, the sequence $\{\Phi_p\}$ epi-converges to the
indicator $\delta_{[-1,1]^n}$ of the box $[-1,1]^n$: it vanishes
uniformly on the box and diverges pointwise outside. Consequently,
$U_{p,\alpha}$ epi-converges to
$U_{\Box,\alpha} + \delta_{[-1,1]^n}$.
\end{proposition}

\noindent Thus $U_{p,\alpha}$ is a smooth bridge from the quartic
soft-spin picture to the hard-box problem selected by bSB saturation.

The hard-box problem admits a direct discrete interpretation via the
following algebraic identity.

\begin{lemma}
\label{lem:vertex}
For every spin vertex $\mathbf{v} \in \{-1,1\}^n$,
\begin{equation}
U_{\Box,\alpha}(\mathbf{v}) = \tfrac{1 - \alpha^2}{2}\, n + E(\mathbf{v}).
\label{eq:vertex}
\end{equation}
Hence the hard-box objective at a vertex differs from the original
Ising energy only by a constant shift that does not depend on the
spin configuration.
\end{lemma}

\noindent If $(\alpha^2 - 1) I + S$ is positive semidefinite, then the
Hessian $\nabla^2 U_{\Box,\alpha} = (1-\alpha^2)I - S$ is negative
semidefinite, so $U_{\Box,\alpha}$ is concave on the box. Physically,
this means that the hard-confinement landscape has no interior basin
that can beat the boundary sign states: a minimizer may be chosen at a
corner of the saturation box. With Lemma~\ref{lem:vertex}, the vertex
minimizers of $U_{\Box,\alpha}$ are therefore exactly the Ising
minimizers. These facts are collected in the following theorem.

\begin{theorem}
\label{thm:bsb_main}
Assume $(\alpha^2 - 1) I + S \succeq 0$.

\begin{description}
\item[(i)] \emph{Hard-box reduction.} $U_{\Box,\alpha}$ is concave on
$[-1,1]^n$. Moreover,
$\arg\min_{\mathbf{x}\in[-1,1]^n} U_{\Box,\alpha}(\mathbf{x})
\cap \{-1,1\}^n \neq \varnothing$, and the vertex minimizers of
$U_{\Box,\alpha}$ are exactly the minimizers of the Ising problem
\eqref{eq:ising}.

\item[(ii)] \emph{Asymptotic readout.} Let $\mathbf{x}_p$ be global
minimizers of $U_{p,\alpha}$ with $p \to \infty$. Assume the following
compactness property: for every spin vector $\mathbf{v} \in \{-1,1\}^n$
such that $\operatorname{sgn}(\mathbf{x}_p) = \mathbf{v}$ for
infinitely many indices $p$, there exists a compact set
$K_{\mathbf{v}} \subset [-1,1]^n$ containing all points of that
recurring subsequence. Then for all sufficiently large $p$,
$\operatorname{sgn}(\mathbf{x}_p)$ is a global minimizer of the
original Ising problem.
\end{description}
\end{theorem}

The first part follows from concavity of the hard-box objective on the
box and the vertex identity above. The same soft-to-hard comparison
gives the asymptotic readout statement in the second part. Under the
compactness hypothesis on recurring sign classes, any subsequential
limit of soft global minimizers is a hard-box minimizer. The concavity
argument then pushes such a limit to a minimizing vertex in the same
sign face, and the vertex identity converts hard-box optimality into
Ising optimality. Hence nonoptimal sign readouts can occur only
finitely often along the soft family.

These statements are variational. They do not prove finite-time
convergence of the full bSB Hamiltonian dynamics with momentum,
clipping, and boundary reset. What they identify is the static
optimization problem associated with the saturated regime: once the
trajectory is effectively confined to the box, the remaining discrete
information is governed by the hard-box objective and by the induced
Ising energy on its vertices.

\subsection{Frozen-set reduction}

The variational results above and the reduction algorithm below play
different roles. The aSB recovery theorem and the bSB hard-box theorem
are global landscape statements: they explain why continuous
simulated-bifurcation landscapes can develop sign branches whose
readout is aligned with the Ising objective, but their sufficient
hypotheses need not hold along the finite benchmark schedules used
below. The algorithmic connection is local in time. Once a set of
coordinates has reached a saturated sign state and satisfies a robust
outward-margin condition, the bSB clipping update keeps those
coordinates frozen in subsequent steps. The direct algebraic step used
by iSTAR is then the frozen-set reduction: once a subset of signs is
fixed, its couplings become an induced external field on the
unresolved subsystem, and the original Ising instance reduces to a
lower-dimensional Ising problem.

For a frozen set $H \subset \{1, \dots, n\}$ with active complement
$Q = H^c$, fixed signs $\mathbf{v}_H \in \{-1,1\}^{|H|}$ on $H$, and a
bSB state at step $k$, define the certified outward drift and the
unresolved coupling envelope by
\begin{equation}
D_i^H(k) := (\alpha_k^2 - 1) + \mu_i v_i + \sum_{j \in H} s_{ij} v_i v_j,
\qquad
B_i^Q := \sum_{j \in Q} |s_{ij}|.
\label{eq:drift_envelope}
\end{equation}
The robust freezing margin is
\begin{equation}
\rho_i(k; H) := D_i^H(k) - B_i^Q .
\label{eq:robust_margin}
\end{equation}
A saturated coordinate with $x_i^{(k)} = v_i \in \{-1,1\}$ and
$y_i^{(k)} = 0$ is certified frozen whenever
$\rho_i(k; H) > 0$.

For diagnostic purposes we also define the one-spin stability gap
\begin{equation}
\Delta_i(\mathbf{v}) := E(\mathbf{v}^{(i)}) - E(\mathbf{v}) = 2 v_i (S \mathbf{v} + \mu)_i,
\label{eq:oneflipgap}
\end{equation}
where $\mathbf{v}^{(i)}$ is $\mathbf{v}$ with the $i$-th coordinate flipped.

\begin{theorem}
\label{thm:frozen}
Let $H \subset \{1, \dots, n\}$ be a frozen set with active complement
$Q = H^c$ and fixed signs $\mathbf{v}_H \in \{-1,1\}^{|H|}$ on $H$.
\begin{description}
\item[(i)] \emph{Algebraic reduction.} Under the constraint that spins
on $H$ are fixed at $\mathbf{v}_H$, minimizing the Ising energy
$E(\mathbf{v})$ over $\mathbf{v} \in \{-1,1\}^n$ is equivalent to
minimizing the reduced energy
$E_Q(\mathbf{v}_Q) := -\tfrac{1}{2} \mathbf{v}_Q^{\mathsf{T}} S_{QQ}
\mathbf{v}_Q - \mu_{\mathrm{eff}} \cdot \mathbf{v}_Q$ over
$\mathbf{v}_Q \in \{-1,1\}^{|Q|}$, where the induced field is
$\mu_{\mathrm{eff}} = \mu_Q + S_{QH} \mathbf{v}_H$.

\item[(ii)] \emph{One-step certified freezing.} If at step $k$ one has
$x_i^{(k)} = v_i \in \{-1,1\}$ and $y_i^{(k)} = 0$ for every
$i \in H$, and the robust freezing condition $\rho_i(k; H) > 0$
holds, then the bSB clipping step yields
$x_i^{(k+1)} = v_i$ and $y_i^{(k+1)} = 0$ for every $i \in H$.
\end{description}
\end{theorem}

\noindent Detailed proofs are given in Theorem~\ref{thm:frozen_app}
and Proposition~\ref{prop:one_step_app};
the persistent version follows by induction on the step
index, see Corollary~\ref{cor:persistent_app}.

The one-spin gap $\Delta_i(\mathbf{v})$ is used as a screening score, not as a
formal certificate. Large positive gaps are consistent with local
stability of the current sign pattern, whereas small gaps mark
coordinates for which a sign change has low discrete cost. The robust
freezing inequalities proved in the appendix are stronger: they
include the current confinement parameter and an $\ell^1$ envelope for
the possible influence of unresolved coordinates. The diagnostic
selector should be read as a practical surrogate for these sufficient
conditions, and its individual components require ablation before one
can separate the contribution of margin, saturation, and momentum
signals.

\subsection{Finite-time freezing guarantee}

The local theorem behind this interpretation is conditional but
finite-time. For a candidate frozen set $H$, active complement $Q$,
and frozen signs $v_i$, the appendix proves that if the bSB state
satisfies $x_i^{(k)} = v_i$ and $y_i^{(k)} = 0$ for $i \in H$, and if
\begin{equation}
(\alpha_k^2 - 1) + \mu_i v_i + \sum_{j \in H} s_{ij} v_i v_j
\;>\;
\sum_{j \in Q} |s_{ij}|
\qquad (i \in H),
\label{eq:certificate_ineq}
\end{equation}
then one further clipped bSB step leaves every coordinate in $H$ at
the same sign with zero momentum.

\begin{corollary}
\label{cor:persistent}
Assume the schedule $\{\alpha_k\}$ is non-decreasing and the same
worst-case box bound $|x_j| \le 1$ continues to hold for every
$j \in Q$ at all subsequent steps. If the certificate inequality
\eqref{eq:certificate_ineq} holds for all $i \in H$ at step $k$, then
$x_i^{(k')} = v_i$ and $y_i^{(k')} = 0$ for every $i \in H$ and every
$k' \ge k$. The proof follows by induction on the step index (see
Appendix~\ref{sec:app_a3}, Corollary~\ref{cor:persistent_app}).
\end{corollary}

\noindent This is the part
of the theory that directly matches the reduction logic: after the
trajectory has entered a robust saturated regime, the frozen
coordinates no longer need to be iterated.

The bSB hard-box theorem should also be read as a sufficient
variational statement, not as a hidden assumption behind the G-set
benchmark. Its positive-semidefinite condition can be restrictive for
raw benchmark matrices: for example, in the G1 instance with the sign
convention used here, $\lambda_{\min}(S) \approx -48.79$, so the
condition $(\alpha^2 - 1)I + S \succeq 0$ would require
$\alpha \gtrsim 7.06$, whereas the benchmark schedule uses
$\alpha \leq 1$. The empirical question is therefore not whether the
global theorem proves the tested trajectories, but whether the
finite-time certificate activates on those trajectories and whether
certified freezing preserves the same-seed baseline. In this sense,
the global variational theory identifies the landscape mechanism, the
frozen-set identity gives the exact reduced Ising subsystem, and the
robust saturated-margin condition supplies the algorithmic bridge.

\section{Methods}
\label{sec:methods}

\subsection{iSTAR algorithm and online certificate}

iSTAR is an algorithmic realization of the frozen-set reduction
principle. It does not treat late-stage stabilization as a heuristic
stopping signal alone. Instead, it uses stabilization to define an
induced Ising subsystem: the fixed coordinates are removed, their
couplings are folded into an effective field, and the remaining
dynamics is run only on the unresolved coordinates. We instantiate
this principle primarily for bSB, where saturation makes the
hard/active split directly observable, and then show that the same
mechanism also appears in aSB and dSB experiments.

\subsubsection{Reduction rule.}

For the primary implementation, iSTAR is online. Along the bSB
trajectory with horizon $K$, the algorithm performs certificate checks
at scheduled intervals. At a check time $k$, it tests the robust
freezing condition for saturated coordinates. Newly certified
coordinates are added to the frozen set $H$, their signs are stored in
$\mathbf{v}_H$, and the active set becomes $Q = H^c$. Their
interaction with the unresolved variables is represented exactly by
the induced field
\begin{equation}
\mu_{\mathrm{eff}} = \mu_Q + S_{QH} \mathbf{v}_H .
\label{eq:induced_field}
\end{equation}
The certified outward drift $D_i^H(k)$, unresolved coupling envelope
$B_i^Q$, and robust freezing margin $\rho_i(k;H) = D_i^H(k) - B_i^Q$
are defined in Section~\ref{sec:theory}
(Eqs.~\eqref{eq:drift_envelope}--\eqref{eq:robust_margin}).
A saturated coordinate with $x_i^{(k)} = v_i \in \{-1,1\}$ and
$y_i^{(k)} = 0$ is certified frozen whenever $\rho_i(k; H) > 0$. The
subsequent dynamics is then continued only on the reduced subsystem
$(S_{QQ}, \mu_{\mathrm{eff}})$, initialized from the current state on
$Q$. If additional coordinates become certified later, the same
reduction step is applied again. The final spin vector is obtained by
combining the frozen signs on $H$ with the terminal signs computed on
the remaining active coordinates.

\begin{table}[h]
\centering
\caption{Online certified iSTAR reduction rule.}
\label{tab:istar_rule}
\small
\begin{tabular}{@{}p{0.18\textwidth}p{0.78\textwidth}@{}}
\toprule
Input &
Interaction matrix $S$, field $\mu$, bSB schedule,
certificate-check interval, and robust freezing thresholds. \\
\midrule
Certificate check &
At each check time $k$, identify saturated coordinates with
$x_i^{(k)} = v_i \in \{-1,1\}$ and $y_i^{(k)} = 0$, and certify those
with $\rho_i(k; H) > 0$. \\
Frozen set update &
Add every newly certified coordinate to $H$, store its sign in
$\mathbf{v}_H$, and set $Q = H^c$. Coordinates not yet certified
remain active. \\
Induced field &
Replace the current active problem by the reduced Ising subsystem
$(S_{QQ}, \mu_{\mathrm{eff}})$ with
$\mu_{\mathrm{eff}} = \mu_Q + S_{QH} \mathbf{v}_H$. \\
Reduced continuation &
Continue the same SB update rule only on $Q$, initialized from the
current state $(\mathbf{x}^{(k)}_Q, \mathbf{y}^{(k)}_Q)$ and driven by
the induced subsystem until the next check or the terminal time. \\
Output &
Return the full spin vector obtained by combining the frozen signs
$\mathbf{v}_H$ with the terminal reduced-tail signs on $Q$. \\
\bottomrule
\end{tabular}
\end{table}

\subsubsection{Active-set selector for diagnostic probe sweeps.}

In the fixed-grid diagnostic sweep reported later, we also use a
practical active-set selector to test how broadly late-tail reduction
is available when a reduction time is chosen from a preassigned grid.
That selector retains coordinates whose readout remains fragile. For a
sign vector $\mathbf{v} = \operatorname{sgn}(\mathbf{x})$, it computes
the one-spin stability gap $\Delta_i(\mathbf{v})$ defined in
Section~\ref{sec:theory} (Eq.~\eqref{eq:oneflipgap}). Small values of $\Delta_i$ indicate
coordinates whose signs remain locally sensitive. This gap is defined
on the Ising energy itself, and therefore the same stability score can
be used for aSB, bSB, and dSB trajectories after sign readout. In the
diagnostic probe-sweep implementation, $Q$ is formed from three
signals: small one-spin margin, incomplete saturation, and large
residual momentum. Concretely, the lowest $15\%$ of one-spin margins
are admitted as initial candidates, coordinates with $|x_i| < 0.98$
are always retained, and the velocity criterion uses the upper $15\%$
quantile of $|y_i|$. The active set is capped at
$\max\{32, \lfloor 0.15 n \rceil\}$ coordinates by keeping the
smallest-margin coordinates.

This diagnostic rule is deliberately conservative but not certified.
It does not assume that the frozen set has been verified by the full
robust inequalities proved in the appendix. Instead, it uses
computable late-stage signals that approximate those conditions along
a trajectory. Its role in this paper is therefore secondary: it
measures the breadth of reducible tails, while the online certified
experiment is the main algorithmic test of the theorem.

\subsection{Solver models}

All experiments use the Ising formulation of Eq.~\eqref{eq:ising}. The
Goto et al.~\href{https://www.science.org/doi/suppl/10.1126/sciadv.abe7953}{Table~S2} G-set instances~\cite{ref22} are loaded with the
sign convention $S = -G$ and the external field is generated with
$\eta_\lambda = 0.2$. The aSB, bSB, and dSB models follow
Eq.~\eqref{eq:pote}, Eq.~\eqref{eq:hardbox}, and the published Tabu-SB
code without the tabu mechanism, respectively. The \href{https://www.science.org/doi/suppl/10.1126/sciadv.abe7953}{Table~S2} protocol
uses the per-instance step counts from Goto et
al.~\cite{ref22}; the all-Gset protocol uses a fixed 1200-step bSB
horizon over the 71 covered G-set graphs, with 20 seeds per graph for
the all-Gset study and 10 seeds for the \href{https://www.science.org/doi/suppl/10.1126/sciadv.abe7953}{Table~S2} study.

\subsection{Benchmark protocol}

We evaluate iSTAR in two complementary ways. The primary certified
experiment is online: along a bSB trajectory, every 50 steps we test
the robust freezing condition for saturated coordinates. Coordinates
that satisfy the certificate are immediately removed from the active
dynamics, their signs are folded into the induced field on the
remaining variables, and the reduced subsystem continues to evolve.
Thus this experiment does not choose a probe point using the full
baseline; it freezes coordinates only when the sufficient condition
has been verified along the trajectory.

We also report a fixed-grid probe sweep as a broader existence test on
the covered all-Gset instances. In that experiment, for each instance
and seed the reference run is the full 1200-step bSB trajectory, and
for each probe index $K_p$ we run the reduced tail from the probe
state and compare the resulting Ising energy with the same-seed
baseline. We report $\Delta E = E_{\mathrm{red}} - E_{\mathrm{base}}$;
since the Ising objective is minimized, a non-degrading reduced run
has $\Delta E \leq 0$.

Compute is measured by the dense matrix-vector work proxy
\begin{equation}
\mathrm{FLOPs}_{\mathrm{base}} = 2 n^2 K,
\qquad
\mathrm{FLOPs}_{\mathrm{red}} = 2 n^2 K_p + 2 |Q|^2 (K - K_p).
\label{eq:flops}
\end{equation}
For progressive online freezing, the same proxy is accumulated over
stages as $2 |Q_k|^2$ per step, where $Q_k$ is the current active set.
This proxy counts the dominant dense interaction work and is
independent of a particular implementation or hardware backend. It
should not be read as a wall-clock measurement; direct runtime
benchmarking is left for a hardware-specific evaluation. The proxy
also does not include the overhead of active-set selection, submatrix
extraction, memory allocation, or batching effects. For the probe
sweep, each graph is summarized by the probe point with the largest
dense-work saving among the non-degrading probes. Complete
per-instance records are provided in the appendix.

\subsection{Classical-solver baseline and data availability}

To address the natural question of how the frozen-set reduction
interacts with other classical Ising/MaxCut solvers, we additionally
compare three representative G-set instances (G1, G22, G43) at the
paper's standard field strength $\eta_\lambda = 0.2$ across six
configurations: full bSB and its iSTAR-reduced form (the primary
solver); full SimCIM and its iSTAR-reduced form (to test whether the
reduction idea transfers to a different continuous Ising dynamics);
and full Tabu-SB discrete SB together with iSTAR-reduced dSB (the same
frozen-set reduction applied to the published Tabu-SB code without its
tabu mechanism). Each configuration is run over five seeds. Cuts are
computed on the modified Ising objective with field, so the absolute
values exceed the classical zero-field best-known cuts of those
instances; the relevant comparison is across methods at the same field
strength. FLOPs are reported per trajectory, normalized to full bSB at
the same instance.

The G-set benchmark instances are publicly available from the Stanford
G-set collection~\cite{ref30}. The scripts used to run the iSTAR
reduction benchmarks, generate the figures, and reproduce the tables
will be made available in a public repository upon publication.

\section{Results}
\label{sec:results}

\subsection{Online certified reduction on the G-set benchmark}

The online experiment is designed to close the gap between the global
variational statements and finite-time computation. The large-$\alpha$
and hard-box results are sufficient landscape theorems and are not
used as assumptions in the benchmark. Instead, the online
implementation evaluates the finite-time robust freezing certificate
directly on the evolving bSB trajectory. Thus the experiment does not
infer stability from the positive-semidefinite condition or from a
retrospectively selected probe point: a coordinate is removed only
after the certificate has been verified for the current state.

We first test the progressive certified version on the Goto et
al.~\href{https://www.science.org/doi/suppl/10.1126/sciadv.abe7953}{Table~S2} G-set protocol~\cite{ref22}, using the per-instance step
counts from that benchmark and ten seeds per graph. We run both the
zero-field Max-Cut setting and a nonzero-field setting generated with
$\eta_\lambda = 0.2$ and field seed $20260422$. Each test covers
G1--G54, for a total of $540$ seed--instance runs. In both settings,
every online reduced run matches the same-seed full bSB baseline
energy (Table~\ref{tab:online}). We then run the same online
certified rule on the broader 71-graph all-Gset evaluation with 20
seeds per graph and a fixed 1200-step bSB horizon.

\begin{table}[h]
\centering
\caption{Online progressive certified iSTAR on the Goto et al.~\href{https://www.science.org/doi/suppl/10.1126/sciadv.abe7953}{Table~S2}
G1--G54 protocol~\cite{ref22}. Coordinates are frozen as soon as they
satisfy the robust freezing certificate; no baseline information is
used to choose the freezing time.}
\label{tab:online}
\small
\begin{tabular}{lcc}
\toprule
Statistic & $\mu = 0$ & $\eta_\lambda = 0.2$ \\
\midrule
Instances & G1--G54 & G1--G54 \\
Seeds per graph & 10 & 10 \\
Certificate check interval & 50 steps & 50 steps \\
Triggered runs & 230/540 & 540/540 \\
Baseline-degrading runs & 0/540 & 0/540 \\
Mean first freeze step & 1313.91 & 255.09 \\
Mean number of freeze events & 1.97 & 36.26 \\
Mean final frozen ratio & 96.73\% & 99.53\% \\
Mean final active ratio & 3.27\% & 0.47\% \\
Mean dense-work saving & 11.47\% & 56.32\% \\
Online better/tie/worse than first-trigger readout & 73/467/0 & 487/38/15 \\
Online better/tie/worse than full baseline & 0/540/0 & 0/540/0 \\
\bottomrule
\end{tabular}
\end{table}

The broader all-Gset online experiment confirms that the certified
rule is not limited to the G1--G54 \href{https://www.science.org/doi/suppl/10.1126/sciadv.abe7953}{Table~S2} subset. With
$\eta_\lambda = 0.2$, all 1420 seed--instance runs trigger certified
freezing and none degrades relative to the same-seed full bSB
baseline. The mean first freeze step is 176.90, the mean final active
ratio is 0.56\%, and the mean dense-work saving is 64.44\%
(Fig.~\ref{fig:overview}).

\begin{figure}[!htbp]
\centering
\includegraphics[width=0.82\textwidth]{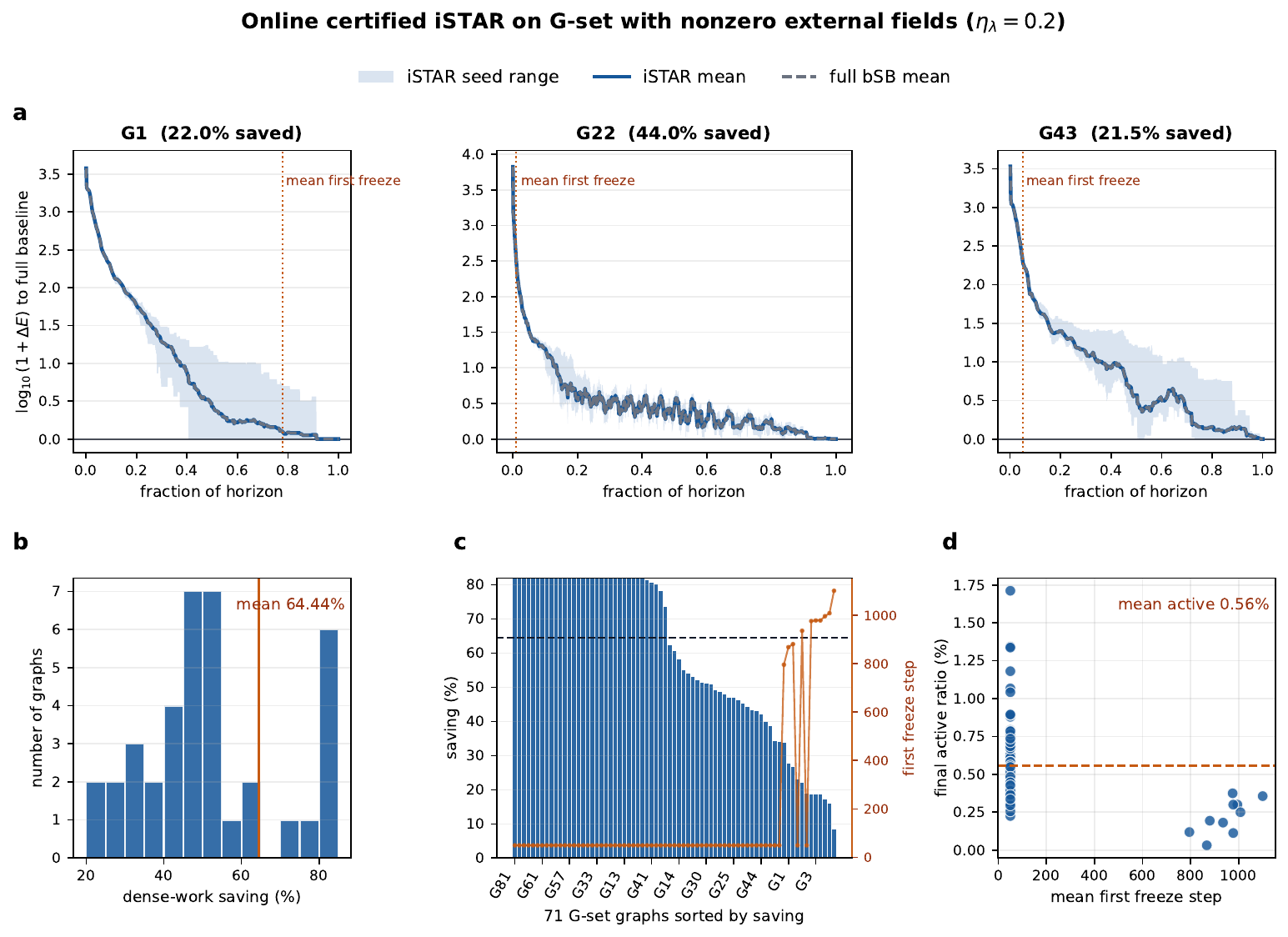}
\caption{Online certified iSTAR with nonzero external fields
($\eta_\lambda = 0.2$). (a) Representative energy gaps
$\log_{10}(1 + \Delta E)$ for G1, G22, and G43 relative to the
same-seed full bSB baseline. (b) Graph-level dense-work savings over
the 71 covered G-set instances. (c) Sorted instance-wise savings with
the mean first freeze step. (d) Mean first freeze step versus final
active ratio.}
\label{fig:overview}
\end{figure}

\begin{figure}[!htbp]
\centering
\includegraphics[width=0.84\textwidth]{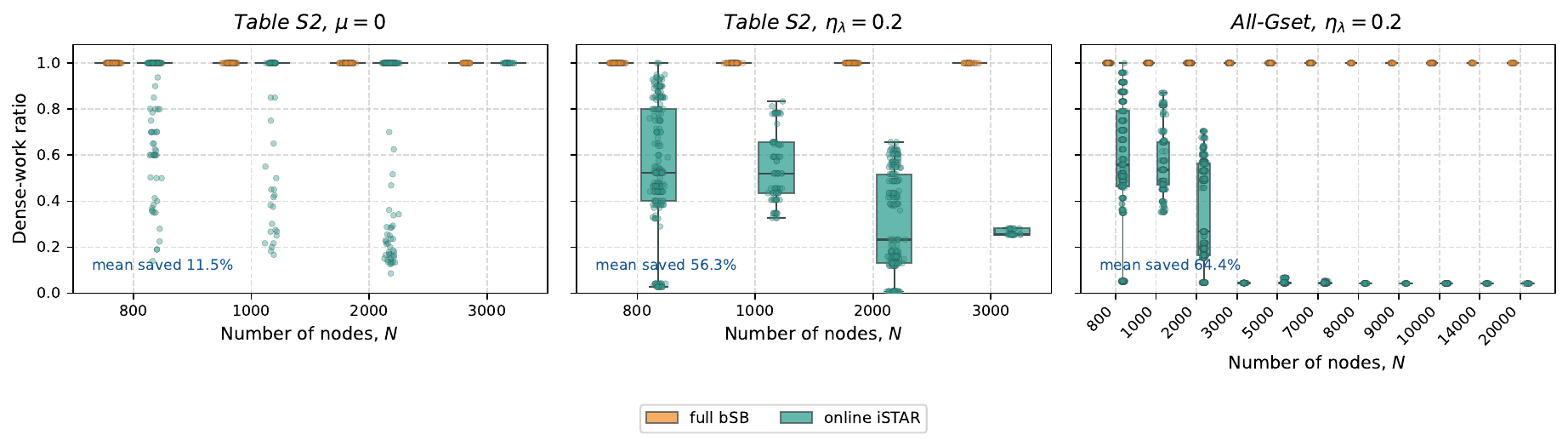}
\caption{Scale-grouped compute reduction for online certified iSTAR.
Each box contains seed--instance runs with the same number of nodes
$N$. Orange denotes the full bSB baseline; teal denotes online iSTAR.
The panels report zero-field \href{https://www.science.org/doi/suppl/10.1126/sciadv.abe7953}{Table~S2}, nonzero-field \href{https://www.science.org/doi/suppl/10.1126/sciadv.abe7953}{Table~S2}, and
nonzero-field all-Gset runs.}
\label{fig:boxplots}
\end{figure}

The online result is the closest experimental counterpart to the
finite-time freezing theorem: the algorithm observes the trajectory,
certifies coordinates using the robust condition, freezes only
certified coordinates, and continues the induced active subsystem. In
both settings, the progressive reduced trajectory returns to the
full-baseline energy in every run. Thus the certified tail is not
merely early stopping, nor is it selected by comparing against the
full run after the fact; it is a reduced continuation triggered by a
verified finite-time stability condition. The scaling picture across
the three settings is summarized by the scale-grouped boxplots in
Fig.~\ref{fig:boxplots}.

\subsection{Probe-sweep reduction on G-set}

The fixed-grid probe sweep covers 71 G-set graphs in the present
evaluation. Its purpose is to test how broadly late-tail reduction is
available over a larger instance set, even when the freezing time is
selected retrospectively from a probe grid. All 71 covered graphs have
at least one non-degrading probe point. The average best non-degrading
dense-work saving over these instances is 53.81\%, with median saving
57.02\%. The selected probe is also highly concentrated at the
beginning of the probe window: 58 graphs reach their best
non-degrading point at step 500, while the remaining 13 graphs require
a later probe between steps 600 and 800.

\begin{table}[h]
\centering
\caption{Summary of the all-Gset iSTAR probe-sweep benchmark. A
non-degrading point means that the mean reduced energy is no worse
than the same-seed 1200-step bSB baseline; the selected probe is the
best non-degrading point on the probed grid.}
\label{tab:probe}
\small
\begin{tabular}{lc}
\toprule
Statistic & Value \\
\midrule
Covered G-set graphs & 71 \\
Seeds per graph & 20 \\
Probe steps & 500, 600, 700, 800, 900, 1000, 1100 \\
Graphs with a non-degrading probe & 71 \\
Average best non-degrading dense-work saving & 53.81\% \\
Median best non-degrading saving & 57.02\% \\
Best probe at step 500 & 58 graphs \\
Best probe at steps 600--800 & 13 graphs \\
Graphs without a non-degrading probe & 0 \\
\bottomrule
\end{tabular}
\end{table}

The main empirical point of the probe sweep is the breadth of the
saving over the probed grid. This is not an isolated best-case effect:
almost every covered graph admits a probe at which more than half of
the dense-tail computation can be removed without lowering mean
solution quality relative to the full baseline. The concentration at
step 500 is equally important. It suggests that, in these runs, branch
stabilization has often occurred well before the nominal 1200-step
endpoint, and the remaining trajectory is dominated by late-stage
refinement rather than global branch selection. This interpretation
remains conditional on the probe sweep, whereas the online certified
experiment above provides the corresponding baseline-free freezing
test on the \href{https://www.science.org/doi/suppl/10.1126/sciadv.abe7953}{Table~S2} protocol. The difference between the mean saving
and the median saving comes from a small group of harder instances,
including G2, G15, G22, G29, and G47, for which the earliest
non-degrading probe occurs at step 800 and the selected saving is
therefore 32.58\%.

Aggregated over all probes, the mean FLOPs ratio is 0.6742 and the
mean active-set ratio is 0.15. At the selected operating points, the
reduced trajectory gives worse energy than the same-seed baseline in
only 1 of 1420 seed--instance pairs, with mean energy change
$-0.0035$. This framing is more relevant than a conventional
win/tie/loss score because the design goal is not to improve the
baseline objective, but to preserve it while removing a large fraction
of the tail computation.

\paragraph{Zero-field contrast.}
The zero-field case presents a contrasting picture. Applying the same
diagnostic active-set selector at the same probe step~500 to the
all-Gset instances with $\mu = 0$ yields 1420 seed--instance runs with
13 wins, 241 ties, and 1166 losses relative to the same-seed full bSB
baseline, a mean $\Delta E = +23.79$, and a mean FLOPs ratio of 0.43
(Appendix Table~\ref{tab:istar_probe_mu0}). At probe step~500 the
nonzero-field runs show only 48 degradations and a mean
$\Delta E = +0.14$ at the same FLOPs ratio
(Table~\ref{tab:istar_probe_summary}). The mean active-set size at the
probe is 460, substantially larger than the $\approx 0.15 n$ cap used
in the nonzero-field setting, reflecting the greater residual
fragility of the probe-time readout. In the absence of a
symmetry-breaking external field, many coordinates that appear
near-saturated at the probe can still undergo sign reversals in the
unperturbed full trajectory, and fixing such coordinates prematurely
biases the induced effective field on the active subsystem. This
contrast reinforces the role of the external field in creating a
favorable reduction regime: the same diagnostic reduction at the same
probe grid point removes 57\% of the dense interaction work and
preserves quality in 1372 of 1420 runs when an external field is
present, but preserves quality in only 254 of 1420 runs when the field
is absent. The pattern is consistent with the asymmetric activation
seen in the online certified experiment (Table~\ref{tab:online}),
where the robust certificate triggers universally with an external
field (540/540 runs) but selectively in the zero-field case
(230/540).

\subsection{Cross-variant evidence and external-field early-stopping control}

The same frozen-set mechanism also appears beyond the primary bSB
setting. In a separate G1--G54 evaluation using the step counts from
the Goto et al.~benchmark protocol~\cite{ref22}, we applied the same
active-set reduction idea to aSB, bSB, and dSB trajectories.
Table~\ref{tab:sb_family} summarizes the resulting non-degrading
operating points; the complete instance-wise table is reported in the
appendix. This protocol is intentionally different from the main bSB
experiment: it uses the per-instance step counts from the Goto et
al.~\cite{ref22} \href{https://www.science.org/doi/suppl/10.1126/sciadv.abe7953}{Table~S2} setting and probe ratios of that horizon,
rather than the fixed 1200-step schedule used in the all-Gset bSB
study. This explains why the bSB coverage reported here is 52/54,
while the main fixed-horizon bSB experiment has coverage 71/71.

\begin{table}[h]
\centering
\caption{Cross-variant evidence for frozen-set reduction on G1--G54
under the Goto et al.~\cite{ref22} \href{https://www.science.org/doi/suppl/10.1126/sciadv.abe7953}{Table~S2} protocol. Coverage counts
instances with at least one non-degrading reduced operating point.}
\label{tab:sb_family}
\small
\begin{tabular}{lccc}
\toprule
Variant & Coverage & Average saving & Probe ratios \\
\midrule
aSB & 54/54 & 47.97\% & 0.5, 0.6, 0.7, 0.8, 0.9 \\
bSB & 52/54 & 36.84\% & 0.5, 0.6, 0.7, 0.8, 0.9 \\
dSB & 40/54 & 18.82\% & 0.5, 0.6, 0.7, 0.8, 0.9 \\
\bottomrule
\end{tabular}
\end{table}

These supporting experiments are smaller than the main all-Gset bSB
study and are not meant to replace it. They show that the same
reduction operation---freeze stable signs, induce an external field on
the active coordinates, and continue the tail---can be instantiated
across the standard SB variants, with the strength of the effect
depending on the underlying dynamics and schedule. The weaker dSB
numbers are therefore informative rather than contradictory: the
discrete update and its schedule leave fewer instances with a large,
safely removable tail under this particular reduction rule.

\paragraph{External-field early-stopping control.}
A separate diagnostic control tests whether a reduced tail can add
information beyond simply reading out a fixed-grid reduction state.
This comparison is not part of the original Goto et
al.~\cite{ref22} baseline protocol and is not the online certified
algorithm; it is an internal control for the fixed-grid reduction
study. For each grid point we compare $E_{\mathrm{probe}}$,
$E_{\mathrm{red}}$, and $E_{\mathrm{base}}$, where
$E_{\mathrm{probe}}$ is the Ising energy obtained by sign readout at
the probe, $E_{\mathrm{red}}$ is the energy after the induced reduced
tail, and $E_{\mathrm{base}}$ is the energy of the full same-seed
trajectory for the same SB variant.

We ran this control on three representative G-set instances with
nonzero external field, using the Goto et al.~\cite{ref22} \href{https://www.science.org/doi/suppl/10.1126/sciadv.abe7953}{Table~S2}
step counts, ten seeds per instance, probe ratios
$0.5, 0.6, 0.7, 0.8, 0.9$, and field scale $\eta_\lambda = 0.2$ with
field seed $20260422$. We applied the same control to aSB, bSB, and
dSB. For each instance and variant, we selected the largest-saving
probe whose mean reduced energy was non-degrading relative to the
same-seed full baseline.

The control separates iSTAR from pure early stopping. Across the nine
variant--instance pairs, the selected reduced tail is non-degrading in
all 90 seed-level runs, improves over probe readout in 86 of them, and
removes an average of 42.36\% of the dense interaction work.
Table~\ref{tab:early_stop} gives the variant-level summary.

\begin{table}[h]
\centering
\caption{External-field early-stopping control across SB variants.
Energies are reported relative to the same-seed full baseline for the
corresponding variant. Values are averaged over the selected
non-degrading operating point for G1, G22, and G43, with ten seeds per
instance.}
\label{tab:early_stop}
\small
\begin{tabular}{lccccc}
\toprule
Variant & Saving &
$E_{\mathrm{probe}}-E_{\mathrm{base}}$ &
$E_{\mathrm{red}}-E_{\mathrm{base}}$ &
$E_{\mathrm{red}}-E_{\mathrm{probe}}$ &
$E_{\mathrm{red}}<E_{\mathrm{probe}}$ \\
\midrule
aSB & 48.88\% & $+138.9124$ & $-1018.2089$ & $-1157.1213$ & 30/30 \\
bSB & 45.62\% & $+2.7180$ & $0.0000$ & $-2.7180$ & 28/30 \\
dSB & 32.58\% & $+1.7952$ & $-0.0134$ & $-1.8086$ & 28/30 \\
\midrule
Mean & 42.36\% & $+47.8085$ & $-339.4074$ & $-387.2160$ & 86/90 \\
\bottomrule
\end{tabular}
\end{table}

\subsection{Classical-solver baseline comparison}

The primary bSB evidence above baselines each reduced run against its
same-seed full bSB trajectory. To address the natural question of how
the frozen-set reduction interacts with other classical Ising/MaxCut
solvers, we additionally compare three representative G-set instances
(G1, G22, G43) at the paper's standard field strength
$\eta_\lambda = 0.2$ across six configurations as described in
Section~\ref{sec:methods}. Table~\ref{tab:classical} reports the
per-instance cut values and FLOPs ratios.

\begin{table}[h]
\centering
\caption{Classical-solver baseline comparison at
$\eta_\lambda = 0.2$ on G1, G22, G43 with 5 seeds per cell. Best and
mean cut are reported on the modified Ising objective. FLOPs ratio is
per trajectory, normalized to full bSB on the same instance.}
\label{tab:classical}
\small
\begin{tabular}{lccc}
\toprule
\multicolumn{4}{c}{\emph{G1 ($n = 800$, $m = 19\,176$)}} \\
Method & Best cut & Mean cut & FLOPs ratio \\
\midrule
Full bSB & 15\,154 & 15\,154 & 1.000 \\
iSTAR-reduced bSB & 15\,154 & 15\,154 & 0.430 \\
Full SimCIM & 14\,808 & 14\,754 & 0.064 \\
iSTAR-reduced SimCIM & 15\,045 & 15\,028 & 0.040 \\
Full dSB (Tabu-SB) & 13\,166 & 13\,135 & 6.225 \\
iSTAR-reduced dSB & 13\,160 & 13\,143 & 4.151 \\
\midrule
\multicolumn{4}{c}{\emph{G22 ($n = 2000$, $m = 19\,990$)}} \\
Method & Best cut & Mean cut & FLOPs ratio \\
\midrule
Full bSB & 14\,142 & 14\,142 & 1.000 \\
iSTAR-reduced bSB & 14\,142 & 14\,141 & 0.444 \\
Full SimCIM & 13\,604 & 13\,492 & 0.064 \\
iSTAR-reduced SimCIM & 14\,544 & 14\,466 & 0.030 \\
Full dSB (Tabu-SB) & 11\,960 & 11\,917 & 7.212 \\
iSTAR-reduced dSB & 11\,960 & 11\,925 & 4.713 \\
\midrule
\multicolumn{4}{c}{\emph{G43 ($n = 1000$, $m = 9\,992$)}} \\
Method & Best cut & Mean cut & FLOPs ratio \\
\midrule
Full bSB & 6\,802 & 6\,802 & 1.000 \\
iSTAR-reduced bSB & 6\,802 & 6\,802 & 0.428 \\
Full SimCIM & 6\,511 & 6\,480 & 0.064 \\
iSTAR-reduced SimCIM & 6\,736 & 6\,710 & 0.033 \\
Full dSB (Tabu-SB) & 5\,937 & 5\,910 & 6.387 \\
iSTAR-reduced dSB & 5\,937 & 5\,910 & 4.124 \\
\bottomrule
\end{tabular}
\end{table}

Three observations follow. First, on G1 and G43 the iSTAR-reduced bSB
matches full bSB exactly across all five seeds (best cut identical,
mean cut equal), with a 57\% FLOPs saving; on G22 the reduced run is
at most a single integer below full bSB, attributable to
floating-point noise from the soft-coordinate restart. Second, the
reduction transfers to SimCIM: iSTAR-reduced SimCIM consistently
outperforms full SimCIM (e.g.\ $+237$ on G1, $+940$ on G22),
suggesting that the frozen-set mechanism is not specific to one
dynamics. Third, dSB is consistently weaker than bSB on these
instances under the same nonzero-field setting, consistent with its
published behavior; iSTAR-reduced dSB recovers a modest FLOPs saving but does
not close the quality gap, so the frozen-set reduction amortizes but
does not substitute for the dynamics it is layered onto.

\section{Discussion}
\label{sec:discussion}

This paper develops frozen-set reduction as a variational mechanism
for continuous Ising solvers. The central observation is not only that
late-stage trajectories often contain stabilized coordinates, but that
such coordinates have an exact algebraic role: once their signs are
fixed, they induce an external field on the unresolved subsystem.
Thus a high-dimensional Ising instance can collapse, near the end of a
continuous trajectory, to a smaller tail problem without changing the
discrete objective on the retained coordinates. This perspective is
related in spirit to persistence and backbone ideas in spin systems,
variable-fixing methods in integer optimization, and screening rules
in continuous optimization: in each case one tries to identify
variables whose state is effectively settled and remove them from the
active problem. The distinction here is that the induced-field
identity gives an exact reduced Ising instance for the unresolved tail
of an SB trajectory.

The theory explains this mechanism at two levels. The global
variational results are sufficient landscape statements: the aSB
theorem gives large-parameter recovery for the quartic continuous
landscape with external field, and the bSB theorem gives a
hard-confinement recovery result under a spectral condition. These
results show how continuous minimizers can align with Ising minimizers,
but they are not guarantees for the finite G-set schedules used in the
experiments. The local theory is closer to the algorithm: the
frozen-set identity supplies the exact reduced Ising objective used by
iSTAR, and the robust-margin freezing criteria show that, once a
saturated set has entered the certified regime, the clipped bSB update
keeps that set fixed.

The experiments test these layers separately. The primary online bSB
experiment evaluates the robust freezing certificate along the actual
trajectory, records when it activates, measures the final frozen
ratio, and checks whether the certified induced tail preserves the
same-seed full baseline. On the nonzero-field G1--G54 protocol, the
certificate activates in all 540 seed--instance runs, freezes 99.53\%
of variables on average by the end of the run, and preserves the
full-baseline energy while removing 56.32\% of the dense interaction
work under the benchmark proxy. The all-Gset online experiment
confirms this pattern at scale: all 1420 runs trigger, none degrades,
and the mean saving is 64.44\%. The all-Gset probe sweep is a
complementary existence study: it shows that all 71 covered graphs
admit a non-degrading fixed-grid probe, with average best
non-degrading dense-work saving 53.81\%. This sweep is useful evidence
about the breadth of late-tail reducibility, but it is not the primary
automatic operating point. The supporting aSB and dSB results show
that related active-set behavior is visible beyond one solver variant.

There are important limitations. The variational theory does not yet
prove a complete finite-time branch-selection theorem for the
projected bSB Hamiltonian dynamics with momentum, clipping, and
boundary reset. Moreover, the sufficient spectral hypotheses in the
hard-box theorem are not satisfied by the raw G-set experiments at the
tested schedule parameters. The theory should therefore be read as a
mechanism and consistency result, not as a global guarantee for the
finite-time trajectories in Section~\ref{sec:results}. What is
available at the finite-time level is conditional: if a candidate
frozen set satisfies the robust saturated-margin inequalities, then
subsequent clipped bSB updates keep it frozen. The online certified
experiment uses this condition directly for its freezing decisions,
whereas the broader fixed-grid probe sweep uses computable surrogate
signals, such as saturation, velocity, and one-spin margins, and
chooses the reduction time from a grid after the comparison is made.
This retrospective selection is a drawback of the sweep and is why it
is reported as diagnostic evidence rather than as the main algorithm.
Those surrogate components require ablation before one can separate
their individual contribution. Finally, the present experiments
quantify dense interaction work rather than wall-clock time. This is
the right abstraction for isolating the reduction mechanism, but a
computational deployment should also measure realized runtime on
concrete CPU, GPU, or accelerator implementations, where memory
layout, batching, and sparse structure can change the observed runtime
gain.

The external-field early-stopping control in
Section~\ref{sec:results} addresses one immediate ambiguity: for the
tested representative instances, the reduced tail is not equivalent to
simply stopping at the fixed reduction time. Broader controls remain
important. A full-Gset early-stopping comparison, wall-clock
measurements, and an ablation of the active-set selector over the
margin quantile, saturation threshold, velocity filter, and active-set
cap would separate the frozen-set principle from
implementation-specific choices and clarify how robustly the observed
savings transfer across hardware and instance families.

The most natural next step is an active-set theory for finite-time
branch stabilization. Such a theory would characterize when a set of
coordinates enters a stable basin, remains frozen, and leaves a
reduced subsystem whose solution preserves the final Ising readout.
This would turn the variational reduction principle developed here
from a static and empirically effective mechanism into a dynamical
guarantee for adaptive continuous Ising solvers. Beyond this, the
induced-field identity is not specific to simulated bifurcation: any
continuous Ising solver with sign readout can in principle carry the
same reduction.


\section*{Acknowledgments}
The authors thank colleagues and collaborators for helpful discussions.
This work was partially
supported by the National Natural Science Foundation of China (Grant
No.\ 12101394).

\appendix
\section*{Appendix}

\makeatletter
\@addtoreset{theorem}{section}
\@addtoreset{equation}{section}
\makeatother
\setcounter{theorem}{0}
\setcounter{equation}{0}
\renewcommand{\thetheorem}{\thesection.\arabic{theorem}}
\renewcommand{\theequation}{\thesection.\arabic{equation}}

\newpage
\section{Detailed theoretical statements and proofs}

This appendix records the proof details behind the variational
statements used in the main text, together with the complete
per-instance benchmark records and cross-variant reduction summaries
referenced in the main text. The mechanism these proofs support is
an \emph{algebraic collapse} of the late-stage bSB trajectory onto a
lower-dimensional active subspace, whose saturated coordinates can
then be eliminated exactly by the frozen-set identity. We keep the
organization parallel to the manuscript: first the large-$\alpha$
aSB recovery theorem, then the bSB soft--hard reduction and
asymptotic readout, and finally the frozen-set identity and
finite-time freezing criteria used by iSTAR.

\subsection*{A.1 Large-$\alpha$ aSB recovery}\label{sec:app_a1}

The associated external-field aSB potential is
\[
U_{\mathrm{aSB}}(\mathbf{x}) =
\sum_{i=1}^{n} \tfrac{1}{4} x_i^4
+ \tfrac{\beta - \alpha^2}{2} \mathbf{x}^{\mathsf{T}} \mathbf{x}
- \tfrac{1}{2} \mathbf{x}^{\mathsf{T}} S \mathbf{x}
- \alpha \mu \cdot \mathbf{x},
\qquad \mathbf{x} \in \mathbb{R}^n,
\]
and the Ising problem is
$\min_{\mathbf{v} \in \{-1,1\}^n} E(\mathbf{v}) :=
-\tfrac{1}{2} \mathbf{v}^{\mathsf{T}} S \mathbf{v} - \mu \cdot \mathbf{v}$.

\noindent\textbf{Notation.}
Define the coupling scale $R_* :=
\max_{1 \le i \le n} \sum_{j=1}^{n} |s_{i,j}|$,
the field scale $M_* := \max_{1 \le i \le n} |\mu_i|$, and the
energy-gap scale
\[
\Delta_* :=
\begin{cases}
\min_{\mathbf{v}', \mathbf{v}'' \in \mathcal{C}}
|E(\mathbf{v}') - E(\mathbf{v}'')|, & \text{if } E \text{ is not constant on } \mathcal{C}, \\
1, & \text{if } E \text{ is constant on } \mathcal{C}.
\end{cases}
\]
Define
\[
C_* := n \left( \tfrac{5}{4} + \tfrac{3}{2}|\beta| + \tfrac{3}{2} R_* + M_* \right).
\]
The threshold scales used in the proof are
\[
\alpha_{\mathrm{ball}} := C_*,\qquad
\alpha_{\mathrm{gap}}  := \tfrac{2 C_*}{\Delta_*},\qquad
\alpha_{\mathrm{loc}}  := 5(2 R_* + 2|\beta| + M_* + 1),
\]
\[
\alpha_{\mathrm{bd}}   := \sqrt{R_* + |\beta| + \tfrac{M_*}{2}},\qquad
\alpha_{\mathrm{sep}}  := \sqrt{3 + |\beta| + R_*}.
\]
The overall threshold in Theorem~\ref{thm:asb_main} is
$\alpha_* := \max\{\alpha_{\mathrm{ball}}, \alpha_{\mathrm{gap}},
\alpha_{\mathrm{loc}}, \alpha_{\mathrm{bd}}, \alpha_{\mathrm{sep}}\}$.

\begin{lemma}\label{lem:local_well}
Fix $\mathbf{v} \in \mathcal{C}$. For every
$\mathbf{x} \in B(\alpha \mathbf{v}, 1)$,
\[
U_{\mathrm{aSB}}(\mathbf{x}) - U_{\mathrm{aSB}}(\alpha \mathbf{v})
\ge -C_* \alpha. \tag{S1}\label{eq:S1}
\]
Moreover, for every
$\mathbf{x} \in \partial B(\alpha \mathbf{v}, 1)$,
\[
U_{\mathrm{aSB}}(\mathbf{x}) - U_{\mathrm{aSB}}(\alpha \mathbf{v})
\ge \alpha^2 - C_* \alpha. \tag{S2}\label{eq:S2}
\]
Consequently, if $\alpha > \alpha_{\mathrm{ball}}$, then the
restriction of $U_{\mathrm{aSB}}$ to $\overline{B(\alpha \mathbf{v}, 1)}$
attains its minimum at some point
$\mathbf{x}_{\mathbf{v}}^* \in B(\alpha \mathbf{v}, 1)$, and
\[
|U_{\mathrm{aSB}}(\mathbf{x}_{\mathbf{v}}^*) - U_{\mathrm{aSB}}(\alpha \mathbf{v})|
\le C_* \alpha. \tag{S3}\label{eq:S3}
\]
\end{lemma}

\noindent\textit{Proof.}
Write $\mathbf{x} = \alpha \mathbf{v} + \delta$ with
$|\delta_i| \le 1$ for every $i$. Direct expansion gives
\begin{align*}
U_{\mathrm{aSB}}(\mathbf{x}) - U_{\mathrm{aSB}}(\alpha \mathbf{v})
&=
\sum_{i=1}^{n} \left[ \alpha^2 \delta_i^2 + \beta (\alpha v_i) \delta_i
+ (\alpha v_i) \delta_i^3 + \tfrac{1}{4} \delta_i^4
+ \tfrac{\beta}{2} \delta_i^2 \right] \\
&\quad - \delta^{\mathsf{T}} S (\alpha \mathbf{v})
- \tfrac{1}{2} \delta^{\mathsf{T}} S \delta
- \alpha \mu \cdot \delta.
\end{align*}
Discard the nonnegative term
$\alpha^2 \sum_i \delta_i^2$. Since $|\delta_i| \le 1$ and
$\alpha > 1$,
\[
\sum_{i=1}^{n} |\beta (\alpha v_i) \delta_i| \le n |\beta| \alpha,
\qquad
\sum_{i=1}^{n} \left| (\alpha v_i) \delta_i^3 \right| \le n \alpha,
\]
\[
\sum_{i=1}^{n} \tfrac{1}{4} |\delta_i|^4 \le \tfrac{n}{4} \alpha,
\qquad
\sum_{i=1}^{n} \tfrac{|\beta|}{2} \delta_i^2 \le \tfrac{n}{2} |\beta| \alpha.
\]
For the coupling terms,
\begin{align*}
|\delta^{\mathsf{T}} S (\alpha \mathbf{v})|
&= \alpha \left| \sum_{i,j} \delta_i s_{i,j} v_j \right|
\le \alpha \sum_i |\delta_i| \sum_j |s_{i,j}| \le n R_* \alpha,
\\
\tfrac{1}{2} |\delta^{\mathsf{T}} S \delta| &\le \tfrac{n}{2} R_* \alpha,
\qquad
\alpha |\mu \cdot \delta| \le \alpha \sum_i |\mu_i| |\delta_i| \le n M_* \alpha.
\end{align*}
Therefore
\[
U_{\mathrm{aSB}}(\mathbf{x}) - U_{\mathrm{aSB}}(\alpha \mathbf{v})
\ge - n \left( \tfrac{5}{4} + \tfrac{3}{2}|\beta| + \tfrac{3}{2} R_* + M_* \right) \alpha
= - C_* \alpha.
\]
This proves \eqref{eq:S1}. If now $|\delta| = 1$, the positive term
$\alpha^2 \sum_i \delta_i^2 = \alpha^2$ appears, and \eqref{eq:S2} follows.
Since $\overline{B(\alpha \mathbf{v}, 1)}$ is compact,
$U_{\mathrm{aSB}}$ attains its minimum there at some point
$\mathbf{x}_{\mathbf{v}}^*$. By \eqref{eq:S1}, such a minimizer cannot lie on
the boundary once $\alpha > \alpha_{\mathrm{ball}}$. Hence
$\mathbf{x}_{\mathbf{v}}^* \in B(\alpha \mathbf{v}, 1)$.
Minimality gives $U_{\mathrm{aSB}}(\mathbf{x}_{\mathbf{v}}^*) \le
U_{\mathrm{aSB}}(\alpha \mathbf{v})$, while \eqref{eq:S1} gives
$U_{\mathrm{aSB}}(\mathbf{x}_{\mathbf{v}}^*) -
U_{\mathrm{aSB}}(\alpha \mathbf{v}) \ge -C_* \alpha$, yielding \eqref{eq:S3}. $\square$

\begin{lemma}\label{lem:localization}
Suppose $\mathbf{x}$ is a local minimum point of
$U_{\mathrm{aSB}}$. If
$\alpha > \max\{\alpha_{\mathrm{loc}}, \alpha_{\mathrm{bd}}, \alpha_{\mathrm{sep}}\}$,
then
\[
|x_i - \alpha \operatorname{sgn}(x_i)| < 1, \qquad i = 1, \dots, n.
\tag{S4}\label{eq:S4}
\]
\end{lemma}

\noindent\textit{Proof.}
Since $\mathbf{x}$ is a local minimum, it is a critical point:
$\nabla U_{\mathrm{aSB}}(\mathbf{x}) = 0$. The $i$th component of the
gradient is
\[
\partial_{x_i} U_{\mathrm{aSB}}(\mathbf{x})
= x_i^3 + (\beta - \alpha^2) x_i - \sum_{j=1}^{n} s_{i,j} x_j - \alpha \mu_i.
\]
Therefore
\[
x_i^3 - \alpha^2 x_i
= \sum_{j=1}^{n} s_{i,j} x_j - \beta x_i + \alpha \mu_i,
\]
and factoring the left-hand side,
\[
x_i (x_i - \alpha)(x_i + \alpha)
= \sum_{j=1}^{n} s_{i,j} x_j - \beta x_i + \alpha \mu_i. \tag{S5}\label{eq:S5}
\]
Choose $m \in \{1, \dots, n\}$ with
$|x_m| = \max_i |x_i|$. If $|x_m| \ge 2 \alpha$, then
$|x_m - \alpha| \ge \alpha$ and $|x_m + \alpha| \ge \alpha$, so
\[
|x_m (x_m - \alpha)(x_m + \alpha)| \ge |x_m| \alpha^2.
\]
On the other hand,
\begin{align*}
|x_m (x_m - \alpha)(x_m + \alpha)|
&\le \sum_j |s_{m,j}| |x_j| + |\beta| |x_m| + \alpha M_* \\
&\le R_* |x_m| + |\beta| |x_m| + \tfrac{M_*}{2} |x_m|,
\end{align*}
because $\alpha M_* \le \tfrac{M_*}{2} |x_m|$ when
$|x_m| \ge 2\alpha$. Since
$R_* + |\beta| + \tfrac{M_*}{2} < \alpha^2$, this cannot hold.
Therefore $|x_i| < 2 \alpha$ for every $i$.

Returning to \eqref{eq:S5}, we obtain for every $i$
\[
|x_i (x_i - \alpha)(x_i + \alpha)|
\le R_* \max_j |x_j| + |\beta| |x_i| + \alpha M_*
< (2 R_* + 2|\beta| + M_*) \alpha
< (2 R_* + 2|\beta| + M_* + 1) \alpha.
\]
Set $B_* := 2 R_* + 2|\beta| + M_* + 1$. We claim that for each $i$,
\[
\min\{|x_i + \alpha|, |x_i|, |x_i - \alpha|\} < \tfrac{5 B_*}{\alpha}.
\]
Indeed, if the minimum were $\ge 5 B_* / \alpha$, then \eqref{eq:S5} implies
$x_i \in (-2\alpha, 2\alpha)$, so two of $|x_i + \alpha|, |x_i|,
|x_i - \alpha|$ are at least $\alpha/2$, giving
\[
|x_i (x_i - \alpha)(x_i + \alpha)| \ge \tfrac{5 B_*}{\alpha} \cdot
\tfrac{\alpha}{2} \cdot \tfrac{\alpha}{2}
= \tfrac{5}{4} B_* \alpha > B_* \alpha,
\]
contradicting the bound above. Therefore
$\min\{|x_i + \alpha|, |x_i|, |x_i - \alpha|\} < 5 B_* / \alpha < 1$.

It remains to exclude the branch near 0. If $|x_i| < 1$, the
diagonal entry of the Hessian at $\mathbf{x}$,
$D^2 U_{\mathrm{aSB}}(\mathbf{x})
= \operatorname{diag}(3 x_1^2 + \beta - \alpha^2, \dots,
3 x_n^2 + \beta - \alpha^2) - S$,
satisfies
\[
3 x_i^2 + \beta - \alpha^2 - s_{i,i}
< 3 + |\beta| + R_* - \alpha^2 < 0,
\]
contradicting the necessary positive-semidefiniteness of the Hessian
at a local minimum. Therefore the branch near 0 cannot occur, and
\eqref{eq:S4} follows. $\square$

\begin{lemma}\label{lem:lifted_id_app}
For any $\mathbf{v}, \mathbf{v}' \in \mathcal{C}$,
\[
U_{\mathrm{aSB}}(\alpha \mathbf{v}) - U_{\mathrm{aSB}}(\alpha \mathbf{v}')
= \alpha^2 (E(\mathbf{v}) - E(\mathbf{v}')).
\tag{S6}\label{eq:S6}
\]
\end{lemma}

\noindent\textit{Proof.}
For any $\mathbf{v} \in \mathcal{C}$,
\begin{align*}
U_{\mathrm{aSB}}(\alpha \mathbf{v})
&= \tfrac{n}{4} \alpha^4
+ \tfrac{\beta - \alpha^2}{2} \alpha^2 \mathbf{v}^{\mathsf{T}} \mathbf{v}
- \tfrac{1}{2} \alpha^2 \mathbf{v}^{\mathsf{T}} S \mathbf{v}
- \alpha^2 \mu \cdot \mathbf{v} \\
&= \tfrac{n}{4} \alpha^4
+ \tfrac{n}{2} (\beta - \alpha^2) \alpha^2
+ \alpha^2 E(\mathbf{v}),
\end{align*}
since $\mathbf{v}^{\mathsf{T}} \mathbf{v} = n$. Subtracting the same
identity for $\mathbf{v}'$ gives \eqref{eq:S6}. $\square$

\begin{proposition}\label{prop:global_align}
Let $\mathbf{x}_0$ be a global minimum point of $U_{\mathrm{aSB}}$ and
set $\mathbf{v}_0 = \operatorname{sgn}(\mathbf{x}_0)$. Under the
hypotheses of Theorem~\ref{thm:asb_main}, the point
$\mathbf{x}_{\mathbf{v}_0}^*$ from
Lemma~\ref{lem:local_well} is also a global minimum point of $U_{\mathrm{aSB}}$.
\end{proposition}

\noindent\textit{Proof.}
Since every global minimum is a local minimum, Lemma~\ref{lem:localization} yields
$|x_{0,i} - \alpha \operatorname{sgn}(x_{0,i})| < 1$ for every $i$.
Hence $\mathbf{x}_0 \in B(\alpha \mathbf{v}_0, 1)$. By the defining
property of $\mathbf{x}_{\mathbf{v}_0}^*$,
$U_{\mathrm{aSB}}(\mathbf{x}_{\mathbf{v}_0}^*) \le U_{\mathrm{aSB}}(\mathbf{x}_0)$.
On the other hand, $\mathbf{x}_0$ is a global minimum, so
$U_{\mathrm{aSB}}(\mathbf{x}_0) \le U_{\mathrm{aSB}}(\mathbf{x})$ for
every $\mathbf{x} \in \mathbb{R}^n$. Therefore
$\mathbf{x}_{\mathbf{v}_0}^*$ is also a global minimum point. $\square$

\begin{proof}[Proof of Theorem~\ref{thm:asb_main}]
Let $\mathbf{x}_0$ be a global minimum point of $U_{\mathrm{aSB}}$ and
set $\mathbf{v}_0 = \operatorname{sgn}(\mathbf{x}_0)$. By
Proposition~\ref{prop:global_align}, the point $\mathbf{x}_{\mathbf{v}_0}^*$ from
Lemma~\ref{lem:local_well} is also a global minimum point. Assume by contradiction
that $\mathbf{v}_0$ is not an Ising minimizer. Then there exists
$\mathbf{v}' \in \mathcal{C}$ with $E(\mathbf{v}') < E(\mathbf{v}_0)$.
By the definition of $\Delta_*$,
$E(\mathbf{v}_0) - E(\mathbf{v}') \ge \Delta_* > 0$.
Let $\mathbf{x}_{\mathbf{v}'}^*$ be the minimizer on
$B(\alpha \mathbf{v}', 1)$. Using \eqref{eq:S3} and \eqref{eq:S6},
\begin{align*}
U_{\mathrm{aSB}}(\mathbf{x}_{\mathbf{v}_0}^*)
- U_{\mathrm{aSB}}(\mathbf{x}_{\mathbf{v}'}^*)
&\ge U_{\mathrm{aSB}}(\alpha \mathbf{v}_0)
- U_{\mathrm{aSB}}(\alpha \mathbf{v}') - 2 C_* \alpha \\
&= \alpha^2 (E(\mathbf{v}_0) - E(\mathbf{v}')) - 2 C_* \alpha \\
&\ge \alpha (\alpha \Delta_* - 2 C_*) > 0,
\end{align*}
because $\alpha > \alpha_* \ge \alpha_{\mathrm{gap}} = 2 C_* / \Delta_*$.
The strict positivity says that $\mathbf{x}_{\mathbf{v}_0}^*$ has
strictly larger energy than another admissible point, contradicting
its global minimality. Therefore no such $\mathbf{v}'$ exists, and
$\mathbf{v}_0$ is a minimizer of the Ising problem.
\end{proof}

\subsection*{A.2 bSB soft--hard reduction and asymptotic readout}

\begin{lemma}\label{lem:epi}
Define
$\iota_{[-1,1]^n}(\mathbf{x}) :=
0$ if $\mathbf{x} \in [-1,1]^n$ and $+\infty$ otherwise. Then, as
$p \to \infty$,
$\Phi_p \xrightarrow{\text{epi}} \iota_{[-1,1]^n}$ on $\mathbb{R}^n$.
\end{lemma}

\noindent\textit{Proof.}
We verify the two defining inequalities of epi-convergence.

\noindent\textit{(i) Lower bound.}
Let $\mathbf{x}_p \to \mathbf{x}$ in $\mathbb{R}^n$. We must show
$\liminf_{p \to \infty} \Phi_p(\mathbf{x}_p) \ge \iota_{[-1,1]^n}(\mathbf{x})$.
If $\mathbf{x} \in [-1,1]^n$ then $\iota_{[-1,1]^n}(\mathbf{x}) = 0$
and the claim follows from $\Phi_p \ge 0$. If $\mathbf{x} \notin
[-1,1]^n$, choose an index $i_0$ with $|x_{i_0}| > 1$ and set
$\varepsilon := (|x_{i_0}| - 1)/2 > 0$. Since
$x_{p, i_0} \to x_{i_0}$, we have $|x_{p, i_0}| \ge 1 + \varepsilon$
for all large $p$. Hence
$\Phi_p(\mathbf{x}_p) \ge |x_{p, i_0}|^{2p} / (2p)
\ge (1 + \varepsilon)^{2p} / (2p) \to +\infty$.

\noindent\textit{(ii) Recovery sequence.}
Fix $\mathbf{x} \in \mathbb{R}^n$. If $\mathbf{x} \notin [-1,1]^n$
then $\iota_{[-1,1]^n}(\mathbf{x}) = +\infty$ and any convergent
choice is admissible; take $\mathbf{y}_p = \mathbf{x}$. If
$\mathbf{x} \in [-1,1]^n$, take $\mathbf{y}_p = \mathbf{x}$. Then
$0 \le \Phi_p(\mathbf{y}_p) = \Phi_p(\mathbf{x}) =
\sum_i |x_i|^{2p}/(2p) \le n/(2p) \to 0
= \iota_{[-1,1]^n}(\mathbf{x})$. So the recovery inequality follows.

Combining (i) and (ii) gives
$\Phi_p \xrightarrow{\text{epi}} \iota_{[-1,1]^n}$. $\square$

\begin{theorem}\label{thm:uniform_bound}
Fix $\alpha, S, \mu$ and $p \ge 1$. Then for every
$\mathbf{x} \in [-1,1]^n$,
\[
0 \le U_{p,\alpha}(\mathbf{x}) - U_{\Box,\alpha}(\mathbf{x})
= \Phi_p(\mathbf{x}) \le \tfrac{n}{2p}.
\tag{S7}\label{eq:S7}
\]
\end{theorem}

\noindent\textit{Proof.}
The identity $U_{p,\alpha} - U_{\Box,\alpha} = \Phi_p$ is immediate.
For $\mathbf{x} \in [-1,1]^n$, each $|x_i| \le 1$ gives
$|x_i|^{2p} \le 1$, so
$\Phi_p(\mathbf{x}) = \sum_i |x_i|^{2p} / (2p) \le n/(2p)$. $\square$

\begin{lemma}\label{lem:sublevel}
Fix $\alpha, S, \mu$, and $p \ge 1$. Then there exists $R_p > 1$ such
that every global minimizer of $U_{p,\alpha}$ on $\mathbb{R}^n$ lies
in $[-R_p, R_p]^n$.
\end{lemma}

\noindent\textit{Proof.}
Choose a reference vertex $\mathbf{v} \in \{-1,1\}^n$ and let
$M_p := U_{p,\alpha}(\mathbf{v})$. Any global minimizer $\mathbf{x}_p$
satisfies $U_{p,\alpha}(\mathbf{x}_p) \le M_p$. Write
$U_{p,\alpha}(\mathbf{x}) =
\sum_i |x_i|^{2p}/(2p)
+ \tfrac{1 - \alpha^2}{2} \|\mathbf{x}\|^2
- \tfrac{1}{2} \mathbf{x}^{\mathsf{T}} S \mathbf{x}
- \mu \cdot \mathbf{x}$.
Let $\|\mu\|_1 = \sum_i |\mu_i|$ and
$C_\alpha(S) := \tfrac{|1 - \alpha^2|}{2} + \tfrac{\|S\|_{\mathrm{op}}}{2}$.
Using $|\mathbf{x}^{\mathsf{T}} S \mathbf{x}| \le \|S\|_{\mathrm{op}} \|\mathbf{x}\|^2$,
$\|\mathbf{x}\|^2 \le n \|\mathbf{x}\|_\infty^2$, and
$|\mu \cdot \mathbf{x}| \le \|\mu\|_1 \|\mathbf{x}\|_\infty$, we obtain
\[
U_{p,\alpha}(\mathbf{x}) \ge \tfrac{\|\mathbf{x}\|_\infty^{2p}}{2p}
- C_\alpha(S) n \|\mathbf{x}\|_\infty^2
- \|\mu\|_1 \|\mathbf{x}\|_\infty.
\]
The right-hand side tends to $+\infty$ as $t := \|\mathbf{x}\|_\infty \to \infty$,
since $2p \ge 2$ and $t^{2p}/(2p)$ dominates the quadratic and
linear terms. Hence there exists $R_p > 1$ such that
$t^{2p}/(2p) - C_\alpha(S) n t^2 - \|\mu\|_1 t > M_p$ for all
$t \ge R_p$. This gives
$U_{p,\alpha}(\mathbf{x}) > M_p$ whenever
$\|\mathbf{x}\|_\infty \ge R_p$, so no global minimizer lies outside
$[-R_p, R_p]^n$. $\square$

\begin{corollary}\label{cor:boxed_min}
Assume $(\alpha^2 - 1) I + S \succeq 0$. Then the minimum values on
the box satisfy
\[
\inf_{\mathbf{x} \in [-1,1]^n} U_{p,\alpha}(\mathbf{x}) \to
\inf_{\mathbf{x} \in [-1,1]^n} U_{\Box,\alpha}(\mathbf{x})
= \tfrac{1 - \alpha^2}{2} n + \min_{\mathbf{v} \in \{-1,1\}^n} E(\mathbf{v}).
\tag{S8}\label{eq:S8}
\]
Moreover, if $\mathbf{x}_p$ is a global minimizer of $U_{p,\alpha}$ on
$\mathbb{R}^n$ and if
$U_{\Box,\alpha}(\operatorname{sgn}(\mathbf{x}_p)) = U_{\Box,\alpha}(\mathbf{x}_p)$,
then for every $\mathbf{w} \in \{-1,1\}^n$,
\[
U_{\Box,\alpha}(\operatorname{sgn}(\mathbf{x}_p)) - U_{\Box,\alpha}(\mathbf{w})
\le \tfrac{n}{2p}.
\tag{S9}\label{eq:S9}
\]
In particular, let
$\Delta_* := \min\{E(\mathbf{v}) - E(\mathbf{w}) : E(\mathbf{w}) < E(\mathbf{v}),\,
\mathbf{v}, \mathbf{w} \in \{-1,1\}^n\}$,
with $\Delta_* = 1$ if the set is empty. If $n/(2p) < \Delta_*$,
then $\operatorname{sgn}(\mathbf{x}_p)$ is an Ising minimizer.
\end{corollary}

\noindent\textit{Proof.}
From Theorem~\ref{thm:uniform_bound},
$U_{\Box,\alpha}(\mathbf{x}) \le U_{p,\alpha}(\mathbf{x}) \le
U_{\Box,\alpha}(\mathbf{x}) + n/(2p)$ for $\mathbf{x} \in [-1,1]^n$.
Taking the infimum gives boxed-minimum convergence. The identity
with the Ising energy follows from Theorem~\ref{thm:bsb_main}(i).

Now let $\mathbf{x}_p$ minimize $U_{p,\alpha}$ on $\mathbb{R}^n$. For
any vertex $\mathbf{w} \in \{-1,1\}^n \subset [-1,1]^n$,
$U_{p,\alpha}(\mathbf{x}_p) \le U_{p,\alpha}(\mathbf{w})
= U_{\Box,\alpha}(\mathbf{w}) + n/(2p)$.
Since $U_{\Box,\alpha} \le U_{p,\alpha}$ on $\mathbb{R}^n$, and by
the shadow assumption,
$U_{\Box,\alpha}(\operatorname{sgn}(\mathbf{x}_p)) = U_{\Box,\alpha}(\mathbf{x}_p)$,
we obtain
$U_{\Box,\alpha}(\operatorname{sgn}(\mathbf{x}_p)) \le
U_{p,\alpha}(\mathbf{x}_p) \le U_{\Box,\alpha}(\mathbf{w}) + n/(2p)$,
which is \eqref{eq:S9}.

Assume now $n/(2p) < \Delta_*$. If $\operatorname{sgn}(\mathbf{x}_p)$
were not an Ising minimizer, there would be a spin vector $\mathbf{w}$
with $E(\mathbf{w}) < E(\operatorname{sgn}(\mathbf{x}_p))$. By
definition, $\Delta_* \le E(\operatorname{sgn}(\mathbf{x}_p)) - E(\mathbf{w})$.
The vertex identity for $U_{\Box,\alpha}$ gives
$U_{\Box,\alpha}(\operatorname{sgn}(\mathbf{x}_p)) -
U_{\Box,\alpha}(\mathbf{w})
= E(\operatorname{sgn}(\mathbf{x}_p)) - E(\mathbf{w})$,
so \eqref{eq:S9} implies $\Delta_* \le n/(2p)$, contradicting the assumption.
Hence $\operatorname{sgn}(\mathbf{x}_p)$ is an Ising minimizer. $\square$

\begin{proof}[Proof of Theorem~\ref{thm:bsb_main}(ii)]
Suppose that infinitely many indices $p$ give a non-optimal
$\mathbf{v}_p := \operatorname{sgn}(\mathbf{x}_p)$. Since the spin
set is finite, one non-optimal spin vector $\mathbf{v}$ occurs along
an infinite subsequence $p_k \to \infty$:
$\operatorname{sgn}(\mathbf{x}_{p_k}) = \mathbf{v}$.
By the compactness assumption, $\{\mathbf{x}_{p_k}\}_{k \ge 1}$ is
contained in a compact $K_{\mathbf{v}} \subset [-1,1]^n$, so we may
pass to a subsequence (still $p_k$) with
$\mathbf{x}_{p_k} \to \mathbf{x}^*$.
Since each $\mathbf{x}_{p_k}$ is a global minimizer of
$U_{p_k, \alpha}$ and $U_{p, \alpha} - U_{\Box, \alpha} = \Phi_p$
with $\Phi_p \to 0$ uniformly on $[-1,1]^n$, the limit point
$\mathbf{x}^*$ is a global minimizer of the hard-box problem on
$[-1,1]^n$: for every $\mathbf{y} \in [-1,1]^n$,
$U_{\Box,\alpha}(\mathbf{x}_{p_k}) \le U_{p_k, \alpha}(\mathbf{x}_{p_k})
\le U_{p_k, \alpha}(\mathbf{y})
= U_{\Box,\alpha}(\mathbf{y}) + \Phi_{p_k}(\mathbf{y})$,
and passing to the limit gives
$U_{\Box,\alpha}(\mathbf{x}^*) \le U_{\Box,\alpha}(\mathbf{y})$.

For each coordinate, the recurring sign pattern is preserved in the
limit (weakly): $v_i = 1$ implies $x_i^* \ge 0$, while
$v_i = -1$ implies $x_i^* \le 0$. Therefore $\mathbf{x}^*$ belongs
to the face of $[-1,1]^n$ determined by these sign constraints, and
$\mathbf{v}$ is among the face's vertices.

Since $(\alpha^2 - 1) I + S \succeq 0$, $U_{\Box,\alpha}$ is concave
on the box. Write $\mathbf{x}^*$ as a convex combination of the
vertices of its face: $\mathbf{x}^* = \sum_u \lambda_u \mathbf{u}$,
with $\lambda_u \ge 0$, $\sum_u \lambda_u = 1$, and every vertex
$\mathbf{u}$ with $\lambda_u > 0$ in the same face as $\mathbf{v}$.
Concavity gives
$U_{\Box,\alpha}(\mathbf{x}^*) \ge \sum_u \lambda_u U_{\Box,\alpha}(\mathbf{u})$,
while global minimality gives
$U_{\Box,\alpha}(\mathbf{x}^*) \le U_{\Box,\alpha}(\mathbf{u})$ for
every vertex $\mathbf{u}$. Hence every vertex with positive weight
has the same hard-box value as $\mathbf{x}^*$; in particular,
$U_{\Box,\alpha}(\mathbf{v}) = U_{\Box,\alpha}(\mathbf{x}^*)$, so
$\mathbf{v}$ is a hard-box minimizing vertex. By the vertex identity,
$\mathbf{v}$ is an Ising minimizer, contradicting the choice of
$\mathbf{v}$ as non-optimal. Therefore only finitely many non-optimal
sign readouts can occur.
\end{proof}

\begin{remark}
The above results are static variational statements, not a convergence
proof for the full projected bSB dynamics with clipping and momentum.
They show that the soft-penalty landscapes approach the
hard-confinement landscape and that, under the PSD condition, the
limiting vertex problem agrees with the Ising problem. The finite-time
bridge used by iSTAR is instead the freezing criterion proved in
Section~A.3 and evaluated online in the main text.
\end{remark}

\subsection*{A.3 Frozen-set identity and freezing criteria}\label{sec:app_a3}

\begin{theorem}\label{thm:frozen_app}
Let $H \subset \{1, \dots, n\}$ and write $Q := H^c$. Fix a spin
pattern $\mathbf{v}_H \in \{-1,1\}^{|H|}$ on $H$. For each
$\mathbf{v}_Q \in \{-1,1\}^{|Q|}$, write
$\mathbf{v} = (\mathbf{v}_Q, \mathbf{v}_H) \in \{-1,1\}^n$. Then the
Ising energy admits the decomposition
\[
E(\mathbf{v}) = -\tfrac{1}{2} \mathbf{v}_Q^{\mathsf{T}} S_{QQ} \mathbf{v}_Q
- \mu_{\mathrm{eff}} \cdot \mathbf{v}_Q
+ C_H(\mathbf{v}_H),
\tag{S10}\label{eq:S10}
\]
where
$\mu_{\mathrm{eff}} := \mu_Q + S_{QH} \mathbf{v}_H$ and
$C_H(\mathbf{v}_H) :=
- \tfrac{1}{2} \mathbf{v}_H^{\mathsf{T}} S_{HH} \mathbf{v}_H
- \mu_H \cdot \mathbf{v}_H$.
Consequently, under the constraint that spins on $H$ are fixed at
$\mathbf{v}_H$, minimizing the full Ising energy over
$\mathbf{v} \in \{-1,1\}^n$ is equivalent to minimizing the reduced
Ising energy
$E_Q(\mathbf{v}_Q) :=
- \tfrac{1}{2} \mathbf{v}_Q^{\mathsf{T}} S_{QQ} \mathbf{v}_Q
- \mu_{\mathrm{eff}} \cdot \mathbf{v}_Q$
over $\mathbf{v}_Q \in \{-1,1\}^{|Q|}$.
\end{theorem}

\noindent\textit{Proof.}
Reorder coordinates so that $Q$ comes first and $H$ comes last.
Write
$S = \bigl(\begin{smallmatrix} S_{QQ} & S_{QH} \\ S_{HQ} & S_{HH} \end{smallmatrix}\bigr)$,
$\mu = \bigl(\begin{smallmatrix} \mu_Q \\ \mu_H \end{smallmatrix}\bigr)$,
$\mathbf{v} = \bigl(\begin{smallmatrix} \mathbf{v}_Q \\ \mathbf{v}_H \end{smallmatrix}\bigr)$.
Since $S$ is symmetric, $S_{HQ} = S_{QH}^{\mathsf{T}}$. Therefore
\begin{align*}
\mathbf{v}^{\mathsf{T}} S \mathbf{v}
&= \mathbf{v}_Q^{\mathsf{T}} S_{QQ} \mathbf{v}_Q
+ \mathbf{v}_Q^{\mathsf{T}} S_{QH} \mathbf{v}_H
+ \mathbf{v}_H^{\mathsf{T}} S_{HQ} \mathbf{v}_Q
+ \mathbf{v}_H^{\mathsf{T}} S_{HH} \mathbf{v}_H \\
&= \mathbf{v}_Q^{\mathsf{T}} S_{QQ} \mathbf{v}_Q
+ 2 \mathbf{v}_Q^{\mathsf{T}} S_{QH} \mathbf{v}_H
+ \mathbf{v}_H^{\mathsf{T}} S_{HH} \mathbf{v}_H,
\end{align*}
and $\mu \cdot \mathbf{v} = \mu_Q \cdot \mathbf{v}_Q + \mu_H \cdot \mathbf{v}_H$.
Substituting into the definition of $E(\mathbf{v})$,
\begin{align*}
E(\mathbf{v})
&= -\tfrac{1}{2} \mathbf{v}_Q^{\mathsf{T}} S_{QQ} \mathbf{v}_Q
- \mathbf{v}_Q^{\mathsf{T}} S_{QH} \mathbf{v}_H
- \mu_Q \cdot \mathbf{v}_Q
- \tfrac{1}{2} \mathbf{v}_H^{\mathsf{T}} S_{HH} \mathbf{v}_H
- \mu_H \cdot \mathbf{v}_H \\
&= -\tfrac{1}{2} \mathbf{v}_Q^{\mathsf{T}} S_{QQ} \mathbf{v}_Q
- (\mu_Q + S_{QH} \mathbf{v}_H) \cdot \mathbf{v}_Q
+ C_H(\mathbf{v}_H),
\end{align*}
which is \eqref{eq:S10}. Since $C_H(\mathbf{v}_H)$ depends only on the fixed
spin pattern on $H$, the full minimization with $\mathbf{v}_H$ held
fixed is equivalent to minimizing
$-\tfrac{1}{2} \mathbf{v}_Q^{\mathsf{T}} S_{QQ} \mathbf{v}_Q
- \mu_{\mathrm{eff}} \cdot \mathbf{v}_Q$ over
$\mathbf{v}_Q \in \{-1,1\}^{|Q|}$. $\square$

\begin{proposition}\label{prop:one_step_app}
Consider one bSB Euler step with step size $\tau > 0$:
\begin{align*}
\tilde{y}_i &= y_i^{(k)} + \tau \bigl( (\alpha_k^2 - 1) x_i^{(k)}
+ \mu_i + (S \mathbf{x}^{(k)})_i \bigr), \\
\tilde{x}_i &= x_i^{(k)} + \tau \tilde{y}_i,
\end{align*}
followed by the bSB clipping rule
\[
x_i^{(k+1)} = \begin{cases}
\operatorname{sgn}(\tilde{x}_i), & |\tilde{x}_i| > 1, \\
\tilde{x}_i, & |\tilde{x}_i| \le 1,
\end{cases}
\qquad
y_i^{(k+1)} = \begin{cases}
0, & |\tilde{x}_i| > 1, \\
\tilde{y}_i, & |\tilde{x}_i| \le 1.
\end{cases}
\]
Let $H \subset \{1, \dots, n\}$ and write $Q := H^c$. Assume that
for every $i \in H$,
$x_i^{(k)} = v_i \in \{-1,1\}$, $y_i^{(k)} = 0$, and
\[
(\alpha_k^2 - 1) + \mu_i v_i + \sum_{j \in H} s_{ij} v_i v_j
> \sum_{j \in Q} |s_{ij}|.
\tag{S11}\label{eq:S11}
\]
Equivalently, $\rho_i(k; H) > 0$ in the compact notation of the main
text. Then for every $i \in H$,
$x_i^{(k+1)} = v_i$, $y_i^{(k+1)} = 0$.
\end{proposition}

\noindent\textit{Proof.}
Fix $i \in H$. Since $x_i^{(k)} = v_i$ and $y_i^{(k)} = 0$, the $i$th
component of the Euler step gives
\[
\tilde{y}_i
= \tau \bigl( (\alpha_k^2 - 1) v_i + \mu_i + (S \mathbf{x}^{(k)})_i \bigr).
\]
Multiplying by $v_i$ and splitting the sum over $H$ and $Q$,
\[
v_i \tilde{y}_i
= \tau \biggl( (\alpha_k^2 - 1) + \mu_i v_i
+ \sum_{j \in H} s_{ij} v_i x_j^{(k)}
+ \sum_{j \in Q} s_{ij} v_i x_j^{(k)} \biggr).
\]
For $j \in H$ the hypothesis gives $x_j^{(k)} = v_j$, while for
$j \in Q$ we only use $|x_j^{(k)}| \le 1$. Therefore
\[
v_i \tilde{y}_i
\ge \tau \biggl( (\alpha_k^2 - 1) + \mu_i v_i
+ \sum_{j \in H} s_{ij} v_i v_j
- \sum_{j \in Q} |s_{ij}| \biggr).
\]
By \eqref{eq:S11} the term in parentheses is strictly positive, so
$v_i \tilde{y}_i > 0$. Now
$\tilde{x}_i = v_i + \tau \tilde{y}_i$ gives
$v_i \tilde{x}_i = 1 + \tau v_i \tilde{y}_i > 1$, so
$|\tilde{x}_i| > 1$ and $\operatorname{sgn}(\tilde{x}_i) = v_i$.
The clipping rule yields $x_i^{(k+1)} = v_i$ and $y_i^{(k+1)} = 0$.
Since $i$ was arbitrary, the conclusion holds for every $i \in H$
satisfying the stated robust freezing condition. $\square$

\begin{corollary}\label{cor:set_one_step_app}
Let $H \subset \{1, \dots, n\}$ and write $Q := H^c$. Assume that at
step $k$ one has
$x_i^{(k)} = v_i \in \{-1,1\}$ and $y_i^{(k)} = 0$ for every
$i \in H$, and that the robust freezing inequality \eqref{eq:S11} holds for every
$i \in H$. Then
$x_i^{(k+1)} = v_i$ and $y_i^{(k+1)} = 0$ for every $i \in H$.
\end{corollary}

\noindent\textit{Proof.}
Apply Proposition~\ref{prop:one_step_app} coordinatewise to each $i \in H$. $\square$

\begin{corollary}\label{cor:persistent_app}
Let $i \in \{1, \dots, n\}$ and $k_0 \ge 0$. Assume that for some
$v_i \in \{-1,1\}$ one has
$x_i^{(k_0)} = v_i$ and $y_i^{(k_0)} = 0$. Assume moreover that the
continuation schedule is monotone
$\alpha_{k+1} \ge \alpha_k$ for all $k \ge k_0$, that the unresolved
coordinates satisfy $|x_j^{(k)}| \le 1$ for every $k \ge k_0$, and
that at the initial step $k_0$ the robust freezing margin
$(\alpha_{k_0}^2 - 1) + \mu_i v_i + \sum_{j \in H} s_{ij} v_i v_j
> \sum_{j \in Q} |s_{ij}|$
holds. Then
$x_i^{(k)} = v_i$ and $y_i^{(k)} = 0$ for all $k \ge k_0$.
\end{corollary}

\noindent\textit{Proof.}
For each $k \ge k_0$, monotonicity gives
$\alpha_k^2 \ge \alpha_{k_0}^2$, so the left-hand side of \eqref{eq:S11} at
step $k$ is no smaller than at step $k_0$. Since the unresolved
coordinates remain in the box, the worst-case perturbation bound on
the right-hand side is unchanged. Hence the one-step hypothesis of
Proposition~\ref{prop:one_step_app} remains valid at every subsequent step. Starting
from $x_i^{(k_0)} = v_i$ and $y_i^{(k_0)} = 0$, repeated application
gives the conclusion for all $k \ge k_0$. $\square$

\begin{corollary}\label{cor:set_persistent_app}
Let $H \subset \{1, \dots, n\}$ and $k_0 \ge 0$. Assume that for every
$i \in H$ one has
$x_i^{(k_0)} = v_i \in \{-1,1\}$ and $y_i^{(k_0)} = 0$, that the
continuation schedule is monotone
$\alpha_k \le \alpha_{k+1}$ for all $k \ge k_0$, that
$|x_j^{(k)}| \le 1$ for every unresolved coordinate $j \in Q = H^c$
and every $k \ge k_0$, and that at the initial step $k_0$ the robust freezing
margin \eqref{eq:S11} holds for every $i \in H$. Then
$x_i^{(k)} = v_i$ and $y_i^{(k)} = 0$ for all $i \in H$, $k \ge k_0$.
\end{corollary}

\noindent\textit{Proof.}
Apply Corollary~\ref{cor:persistent_app} coordinatewise to each $i \in H$. $\square$

\begin{remark}
Corollary~\ref{cor:persistent_app} is the mathematical form of the practical freezing
principle used by iSTAR: once a coordinate has entered a sufficiently
strong robust freezing margin regime at a late enough stage, and once the
continuation parameter is still moving in the favorable direction,
that coordinate no longer needs to be updated. In numerical
implementation, however, one usually works with computable surrogate
criteria such as near-saturation, small velocity, and favorable
one-spin freezing margins. These surrogates are weaker than \eqref{eq:S11}; therefore
occasional empirical sign reversals or re-activations are compatible
with the theory and simply indicate that the theorem's safety freezing
margin was not yet fully attained. The set-valued versions in
Corollaries~\ref{cor:set_one_step_app} and~\ref{cor:set_persistent_app} are the direct mathematical prototype
for the hard/soft split used by iSTAR.
\end{remark}

\newpage
\section*{Table A.1. Summary of online certified iSTAR benchmarks}

The online rule freezes coordinates only after the robust freezing
certificate is satisfied along the current bSB trajectory; no same-seed
baseline information is used to choose the freezing time. Dense-work saving is
reported under the benchmark FLOP proxy.

\begin{center}
\small
\begin{table}[!htbp]
\centering
\caption{Summary of the online certified iSTAR benchmarks. The online rule
freezes coordinates only after the robust certificate is satisfied along the
current bSB trajectory; no same-seed baseline information is used to choose the
freezing time. Dense-work saving is reported under the benchmark FLOP proxy.}
\label{tab:appendix_online_certified_summary}
\small
\setlength{\tabcolsep}{3.5pt}
\begin{tabular}{lrrrrrrr}
\toprule
Protocol & Runs & Triggered & Degrading & First freeze & Frozen & Active & Saving \\
\midrule
Table S2, $\mu=0$ & 540 & 230 & 0 & 1313.91 & $96.73\%$ & $3.27\%$ & $11.47\%$ \\
Table S2, $\eta_\lambda=0.2$ & 540 & 540 & 0 & 255.09 & $99.53\%$ & $0.47\%$ & $56.32\%$ \\
All-Gset, $\eta_\lambda=0.2$ & 1420 & 1420 & 0 & 176.90 & $99.44\%$ & $0.56\%$ & $64.44\%$ \\
\bottomrule
\end{tabular}
\end{table}

\end{center}

\newpage
\section*{Table A.2. Instance-wise online certified iSTAR records for the nonzero-field all-Gset benchmark}

Each instance uses 20 seeds, fixed 1200-step bSB horizon,
$\eta_\lambda = 0.2$. All 1420
seed--instance runs trigger certified reduction and none degrades
relative to the same-seed full bSB baseline.

\begin{center}
\footnotesize
\begin{table}[!htbp]
\centering
\caption{Instance-wise online certified iSTAR records for the nonzero-field all-Gset benchmark. Each instance uses 20 seeds, fixed 1200-step bSB horizon, $\eta_\lambda=0.2$, and field seed 20260422. All 1420 seed--instance runs trigger certified reduction and none degrades relative to the same-seed full bSB baseline.}
\label{tab:online_certified_allgset_instance}
\scriptsize
\setlength{\tabcolsep}{2.2pt}
\begin{tabular}{lrrrr@{\hspace{0.9em}}lrrrr}
\toprule
Instance & Trig. & First & Active & Saving & Instance & Trig. & First & Active & Saving \\
\midrule
G1 & 20/20 & 867.50 & 0.03\% & 27.71\% & G37 & 20/20 & 50.00 & 0.88\% & 73.50\% \\
G2 & 20/20 & 1100.00 & 0.36\% & 8.33\% & G38 & 20/20 & 50.00 & 0.56\% & 81.27\% \\
G3 & 20/20 & 977.50 & 0.11\% & 18.54\% & G39 & 20/20 & 50.00 & 0.89\% & 83.64\% \\
G4 & 20/20 & 995.00 & 0.30\% & 17.08\% & G40 & 20/20 & 50.00 & 0.79\% & 84.63\% \\
G5 & 20/20 & 1007.50 & 0.25\% & 16.04\% & G41 & 20/20 & 50.00 & 0.56\% & 80.65\% \\
G6 & 20/20 & 975.00 & 0.38\% & 18.75\% & G42 & 20/20 & 50.00 & 0.52\% & 83.45\% \\
G7 & 20/20 & 935.00 & 0.18\% & 22.08\% & G43 & 20/20 & 50.00 & 0.69\% & 18.89\% \\
G8 & 20/20 & 795.00 & 0.12\% & 33.75\% & G44 & 20/20 & 50.00 & 0.23\% & 42.18\% \\
G9 & 20/20 & 977.50 & 0.30\% & 18.54\% & G45 & 20/20 & 50.00 & 0.26\% & 34.11\% \\
G10 & 20/20 & 880.00 & 0.19\% & 26.66\% & G46 & 20/20 & 50.00 & 0.37\% & 46.31\% \\
G11 & 20/20 & 50.00 & 1.71\% & 95.05\% & G47 & 20/20 & 50.00 & 0.75\% & 23.06\% \\
G12 & 20/20 & 50.00 & 1.34\% & 94.72\% & G48 & 20/20 & 50.00 & 0.50\% & 95.59\% \\
G13 & 20/20 & 50.00 & 1.18\% & 94.66\% & G49 & 20/20 & 50.00 & 0.90\% & 95.54\% \\
G14 & 20/20 & 50.00 & 0.63\% & 58.15\% & G50 & 20/20 & 50.00 & 0.71\% & 95.60\% \\
G15 & 20/20 & 50.00 & 0.58\% & 45.33\% & G51 & 20/20 & 50.00 & 0.30\% & 51.22\% \\
G16 & 20/20 & 50.00 & 0.38\% & 50.91\% & G52 & 20/20 & 50.00 & 0.77\% & 62.21\% \\
G17 & 20/20 & 50.00 & 0.63\% & 52.07\% & G53 & 20/20 & 50.00 & 0.59\% & 54.89\% \\
G18 & 20/20 & 50.00 & 1.34\% & 48.00\% & G54 & 20/20 & 50.00 & 0.79\% & 54.13\% \\
G19 & 20/20 & 50.00 & 0.28\% & 49.11\% & G55 & 20/20 & 50.00 & 0.57\% & 95.68\% \\
G20 & 20/20 & 50.00 & 0.56\% & 53.16\% & G56 & 20/20 & 50.00 & 0.74\% & 95.68\% \\
G21 & 20/20 & 50.00 & 0.79\% & 60.70\% & G57 & 20/20 & 50.00 & 0.58\% & 95.65\% \\
G22 & 20/20 & 50.00 & 0.51\% & 38.63\% & G58 & 20/20 & 50.00 & 0.44\% & 93.31\% \\
G23 & 20/20 & 50.00 & 0.67\% & 43.18\% & G59 & 20/20 & 50.00 & 0.56\% & 95.28\% \\
G24 & 20/20 & 50.00 & 0.40\% & 34.28\% & G60 & 20/20 & 50.00 & 0.48\% & 95.71\% \\
G25 & 20/20 & 50.00 & 0.31\% & 46.97\% & G61 & 20/20 & 50.00 & 0.41\% & 95.72\% \\
G26 & 20/20 & 50.00 & 0.63\% & 47.06\% & G62 & 20/20 & 50.00 & 0.52\% & 95.69\% \\
G27 & 20/20 & 50.00 & 0.54\% & 39.98\% & G63 & 20/20 & 50.00 & 0.44\% & 94.78\% \\
G28 & 20/20 & 50.00 & 0.42\% & 43.12\% & G64 & 20/20 & 50.00 & 0.55\% & 95.41\% \\
G29 & 20/20 & 50.00 & 0.43\% & 44.16\% & G65 & 20/20 & 50.00 & 0.48\% & 95.72\% \\
G30 & 20/20 & 50.00 & 0.31\% & 51.10\% & G66 & 20/20 & 50.00 & 0.45\% & 95.74\% \\
G31 & 20/20 & 50.00 & 0.61\% & 48.53\% & G67 & 20/20 & 50.00 & 0.40\% & 95.75\% \\
G32 & 20/20 & 50.00 & 1.07\% & 95.46\% & G70 & 20/20 & 50.00 & 0.39\% & 95.74\% \\
G33 & 20/20 & 50.00 & 0.50\% & 95.41\% & G72 & 20/20 & 50.00 & 0.43\% & 95.73\% \\
G34 & 20/20 & 50.00 & 1.04\% & 95.32\% & G77 & 20/20 & 50.00 & 0.36\% & 95.76\% \\
G35 & 20/20 & 50.00 & 0.67\% & 80.11\% & G81 & 20/20 & 50.00 & 0.34\% & 95.79\% \\
G36 & 20/20 & 50.00 & 0.67\% & 78.20\% &  &  &  &  &  \\
\bottomrule
\end{tabular}
\end{table}

\end{center}

\newpage
\section*{Table A.3. Instance-wise iSTAR summary on the all-Gset reduction benchmark}

We report the selected probe with maximum dense-work saving among
probes satisfying $\Delta E = E_{\mathrm{red}} - E_{\mathrm{base}} \le 0$.

\begin{center}
\footnotesize
\begin{table}[!htbp]
\centering
\caption{Instance-wise iSTAR summary on the all-Gset reduction benchmark. We report the selected probe with maximum dense-work saving among probes satisfying mean $\Delta E=E_{\mathrm{red}}-E_{\mathrm{base}}\leq0$.}
\label{tab:istar_full_gset}
\scriptsize
\setlength{\tabcolsep}{3pt}
\begin{tabular}{lrrr@{\hspace{1.2em}}lrrr}
\toprule
Instance & Probe & Mean $\Delta E$ & Saving & Instance & Probe & Mean $\Delta E$ & Saving \\
\midrule
G1 & 500 & 0.0000 & 57.02\% & G37 & 500 & 0.0000 & 57.02\% \\
G2 & 800 & 0.0000 & 32.58\% & G38 & 500 & 0.0000 & 57.02\% \\
G3 & 500 & -0.1205 & 57.02\% & G39 & 500 & 0.0000 & 57.02\% \\
G4 & 500 & -0.1258 & 57.02\% & G40 & 500 & 0.0000 & 57.02\% \\
G5 & 600 & 0.0000 & 48.87\% & G41 & 500 & 0.0000 & 57.02\% \\
G6 & 600 & 0.0000 & 48.87\% & G42 & 500 & 0.0000 & 57.02\% \\
G7 & 700 & 0.0000 & 40.73\% & G43 & 600 & 0.0000 & 48.87\% \\
G8 & 500 & 0.0000 & 57.02\% & G44 & 500 & 0.0000 & 57.02\% \\
G9 & 700 & 0.0000 & 40.73\% & G45 & 500 & 0.0000 & 57.02\% \\
G10 & 700 & 0.0000 & 40.73\% & G46 & 500 & 0.0000 & 57.02\% \\
G11 & 500 & 0.0000 & 57.02\% & G47 & 800 & 0.0000 & 32.58\% \\
G12 & 500 & 0.0000 & 57.02\% & G48 & 500 & 0.0000 & 57.02\% \\
G13 & 500 & 0.0000 & 57.02\% & G49 & 500 & 0.0000 & 57.02\% \\
G14 & 500 & 0.0000 & 57.02\% & G50 & 500 & 0.0000 & 57.02\% \\
G15 & 800 & 0.0000 & 32.58\% & G51 & 500 & 0.0000 & 57.02\% \\
G16 & 500 & 0.0000 & 57.02\% & G52 & 700 & 0.0000 & 40.73\% \\
G17 & 500 & 0.0000 & 57.02\% & G53 & 500 & 0.0000 & 57.02\% \\
G18 & 500 & 0.0000 & 57.02\% & G54 & 500 & 0.0000 & 57.02\% \\
G19 & 500 & 0.0000 & 57.02\% & G55 & 500 & 0.0000 & 57.02\% \\
G20 & 500 & 0.0000 & 57.02\% & G56 & 500 & 0.0000 & 57.02\% \\
G21 & 500 & 0.0000 & 57.02\% & G57 & 500 & 0.0000 & 57.02\% \\
G22 & 800 & 0.0000 & 32.58\% & G58 & 500 & 0.0000 & 57.02\% \\
G23 & 500 & 0.0000 & 57.02\% & G59 & 500 & 0.0000 & 57.02\% \\
G24 & 500 & 0.0000 & 57.02\% & G60 & 500 & 0.0000 & 57.02\% \\
G25 & 500 & -0.0039 & 57.02\% & G61 & 500 & 0.0000 & 57.02\% \\
G26 & 700 & 0.0000 & 40.73\% & G62 & 500 & 0.0000 & 57.02\% \\
G27 & 500 & 0.0000 & 57.02\% & G63 & 500 & 0.0000 & 57.02\% \\
G28 & 500 & 0.0000 & 57.02\% & G64 & 500 & 0.0000 & 57.02\% \\
G29 & 800 & 0.0000 & 32.58\% & G65 & 500 & 0.0000 & 57.02\% \\
G30 & 500 & 0.0000 & 57.02\% & G66 & 500 & 0.0000 & 57.02\% \\
G31 & 500 & 0.0000 & 57.02\% & G67 & 500 & 0.0000 & 57.02\% \\
G32 & 500 & 0.0000 & 57.02\% & G70 & 500 & 0.0000 & 57.02\% \\
G33 & 500 & 0.0000 & 57.02\% & G72 & 500 & 0.0000 & 57.02\% \\
G34 & 500 & 0.0000 & 57.02\% & G77 & 500 & 0.0000 & 57.02\% \\
G35 & 500 & 0.0000 & 57.02\% & G81 & 500 & 0.0000 & 57.02\% \\
G36 & 500 & 0.0000 & 57.02\% &  &  &  &  \\
\bottomrule
\end{tabular}
\end{table}

\end{center}

\newpage
\section*{Table A.4. Probe-level iSTAR summary on the 71 covered G-set instances}

Win/tie/loss (W/T/L) counts are seed-level comparisons with the same-seed 1200-step
bSB baseline; wins have $\Delta E < 0$, ties have $\Delta E = 0$, and
losses have $\Delta E > 0$.

\begin{center}
\small
\begin{table}[ht]
\centering
\caption{Probe-level iSTAR summary on the 71 covered G-set instances. W/T/L counts seed-level comparisons with the same-seed 1200-step bSB baseline; wins have $\Delta E<0$, ties have $\Delta E=0$, and losses have $\Delta E>0$.}
\label{tab:istar_probe_summary}
\small
\begin{tabular}{rrrrrr}
\toprule
Probe step & Samples & W/T/L & Mean $\Delta E$ & FLOPs ratio & Saving \\
\midrule
500 & 1420 & 9/1363/48 & 0.1410 & 0.4298 & 57.02\% \\
600 & 1420 & 11/1372/37 & 0.0861 & 0.5112 & 48.88\% \\
700 & 1420 & 1/1409/10 & 0.0235 & 0.5927 & 40.73\% \\
800 & 1420 & 1/1416/3 & 0.0041 & 0.6742 & 32.58\% \\
900 & 1420 & 1/1418/1 & 0.0012 & 0.7556 & 24.44\% \\
1000 & 1420 & 0/1420/0 & 0.0000 & 0.8371 & 16.29\% \\
1100 & 1420 & 0/1420/0 & 0.0000 & 0.9185 & 8.15\% \\
\bottomrule
\end{tabular}
\end{table}

\end{center}

\newpage
\section*{Table A.5. Zero-field ($\mu = 0$) probe sweep on the all-Gset instances}

The diagnostic active-set selector is applied at probe step 500.

\begin{center}
\small
\begin{table}[ht]
\centering
\caption{Zero-field probe sweep ($\mu=0$) on the all-Gset instances with the diagnostic active-set selector at probe step 500.}
\label{tab:istar_probe_mu0}
\small
\begin{tabular}{lc}
\toprule
Statistic & Value \\
\midrule
Probe step & 500 \\
Covered G-set graphs & 71 \\
Seeds per graph & 20 \\
Total samples & 1420 \\
Wins/Ties/Losses vs full baseline & 13/241/1166 \\
Mean $\Delta E$ & $+23.79$ \\
Mean FLOPs ratio & $0.43$ \\
Mean active set size & $460.14$ \\
\bottomrule
\end{tabular}
\end{table}

\end{center}

\newpage
\section*{Table A.6. SB-family cross-variant reduction on G1--G54}

The same tail-reduction principle is applied across aSB, bSB, and dSB
under the Goto et al.~\cite{ref22} Table~S2 protocol on G1--G54. Each method block
reports the best non-degrading operating point, selected by maximum
FLOP saving among probe ratios $0.5, 0.6, 0.7, 0.8, 0.9$ of the
per-instance step count. The entries are probe ratio $\rho$, mean
$\Delta E = E_{\mathrm{red}} - E_{\mathrm{base}}$, and FLOP saving.
A dash means that no probe ratio is non-degrading in mean energy for
that instance and method. Across the 162 instance--variant cells
(54 instances $\times$ 3 variants), 18 are dashes: 0 for aSB,
2 for bSB, and 16 for dSB. The remaining cells report the best
non-degrading operating point as defined above. The dashed cells
correspond to instances for which every probed reduction step ratio
produces a strictly positive mean $\Delta E$, so no non-degrading
operating point exists.

\begin{landscape}
\footnotesize
\scriptsize
\setlength{\tabcolsep}{3pt}
\begin{longtable}{lrrr|rrr|rrr}
\caption{Instance-wise reduction summary for three simulated bifurcation variants under the Goto et al.~Table S2 protocol on G1--G54. Each method block reports the best non-degrading operating point when one exists, selected by maximum FLOP saving among probe ratios $0.5,0.6,0.7,0.8,0.9$ of the per-instance step count. The entries are probe ratio $\rho$, mean $\Delta E=E_{\mathrm{red}}-E_{\mathrm{base}}$, and FLOP saving. A dash means that no probe ratio is non-degrading in mean energy for that instance and method.}\label{tab:sb_family_parallel_reduction}\\
\toprule
Instance & \multicolumn{3}{c|}{aSB} & \multicolumn{3}{c|}{bSB} & \multicolumn{3}{c}{dSB} \\
\cmidrule(lr){2-4}\cmidrule(lr){5-7}\cmidrule(lr){8-10}
 & $\rho$ & Mean $\Delta E$ & Saving & $\rho$ & Mean $\Delta E$ & Saving & $\rho$ & Mean $\Delta E$ & Saving \\
\midrule
\endfirsthead
\toprule
Instance & \multicolumn{3}{c|}{aSB} & \multicolumn{3}{c|}{bSB} & \multicolumn{3}{c}{dSB} \\
\cmidrule(lr){2-4}\cmidrule(lr){5-7}\cmidrule(lr){8-10}
 & $\rho$ & Mean $\Delta E$ & Saving & $\rho$ & Mean $\Delta E$ & Saving & $\rho$ & Mean $\Delta E$ & Saving \\
\midrule
\endhead
\midrule
\multicolumn{10}{r}{Continued on next page} \\
\endfoot
\bottomrule
\endlastfoot
G1 & 0.5 & -109.60 & 48.88\% & 0.6 & -0.20 & 39.10\% & 0.8 & 0.00 & 19.55\% \\
G2 & 0.5 & -156.20 & 48.88\% & 0.5 & 0.00 & 48.88\% & 0.5 & 0.00 & 48.88\% \\
G3 & 0.5 & -130.40 & 48.88\% & 0.8 & 0.00 & 19.55\% & 0.9 & 0.00 & 9.78\% \\
G4 & 0.5 & -44.20 & 48.88\% & 0.9 & 0.00 & 9.78\% & 0.9 & 0.00 & 9.78\% \\
G5 & 0.5 & -125.80 & 48.88\% & 0.8 & 0.00 & 19.55\% & -- & -- & -- \\
G6 & 0.5 & -119.60 & 48.88\% & 0.7 & 0.00 & 29.33\% & 0.9 & 0.00 & 9.78\% \\
G7 & 0.5 & -43.00 & 48.88\% & 0.8 & 0.00 & 19.55\% & 0.8 & 0.00 & 19.55\% \\
G8 & 0.5 & -162.00 & 48.88\% & 0.8 & 0.00 & 19.55\% & 0.9 & 0.00 & 9.78\% \\
G9 & 0.5 & -213.40 & 48.88\% & 0.8 & 0.00 & 19.55\% & 0.8 & 0.00 & 19.55\% \\
G10 & 0.5 & -167.80 & 48.88\% & 0.6 & 0.00 & 39.10\% & -- & -- & -- \\
G11 & 0.5 & -8.00 & 48.88\% & 0.6 & 0.00 & 39.10\% & -- & -- & -- \\
G12 & 0.6 & 0.00 & 39.10\% & 0.5 & 0.00 & 48.88\% & 0.9 & 0.00 & 9.78\% \\
G13 & 0.5 & -6.00 & 48.88\% & 0.5 & 0.00 & 48.88\% & -- & -- & -- \\
G14 & 0.5 & -85.80 & 48.88\% & 0.5 & 0.00 & 48.88\% & 0.9 & -0.20 & 9.78\% \\
G15 & 0.5 & -101.60 & 48.88\% & 0.5 & 0.00 & 48.88\% & 0.9 & -0.20 & 9.78\% \\
G16 & 0.5 & -85.00 & 48.88\% & 0.5 & 0.00 & 48.88\% & 0.5 & 0.00 & 48.88\% \\
G17 & 0.5 & -85.00 & 48.88\% & 0.5 & 0.00 & 48.88\% & -- & -- & -- \\
G18 & 0.5 & -84.40 & 48.88\% & 0.8 & 0.00 & 19.55\% & -- & -- & -- \\
G19 & 0.5 & -141.20 & 48.88\% & 0.7 & 0.00 & 29.33\% & 0.9 & 0.00 & 9.78\% \\
G20 & 0.5 & -10.00 & 48.88\% & 0.9 & 0.00 & 9.78\% & -- & -- & -- \\
G21 & 0.5 & -111.60 & 48.88\% & 0.8 & 0.00 & 19.55\% & -- & -- & -- \\
G22 & 0.5 & -250.80 & 48.88\% & 0.5 & 0.00 & 48.88\% & 0.8 & 0.00 & 19.55\% \\
G23 & 0.5 & -251.20 & 48.88\% & 0.8 & 0.00 & 19.55\% & 0.9 & 0.00 & 9.78\% \\
G24 & 0.5 & -289.00 & 48.88\% & 0.5 & 0.00 & 48.88\% & 0.5 & -0.20 & 48.88\% \\
G25 & 0.5 & -227.60 & 48.88\% & 0.5 & 0.00 & 48.88\% & 0.5 & 0.00 & 48.88\% \\
G26 & 0.5 & -244.80 & 48.88\% & 0.5 & 0.00 & 48.88\% & -- & -- & -- \\
G27 & 0.5 & -221.40 & 48.88\% & 0.8 & 0.00 & 19.55\% & 0.9 & 0.00 & 9.78\% \\
G28 & 0.5 & -209.00 & 48.88\% & -- & -- & -- & 0.9 & 0.00 & 9.78\% \\
G29 & 0.5 & -323.40 & 48.88\% & -- & -- & -- & 0.9 & 0.00 & 9.78\% \\
G30 & 0.5 & -218.00 & 48.88\% & 0.5 & 0.00 & 48.88\% & -- & -- & -- \\
G31 & 0.5 & -291.80 & 48.88\% & 0.5 & 0.00 & 48.88\% & 0.5 & 0.00 & 48.88\% \\
G32 & 0.6 & -10.80 & 39.10\% & 0.5 & 0.00 & 48.88\% & 0.9 & 0.00 & 9.78\% \\
G33 & 0.5 & -4.80 & 48.88\% & 0.5 & 0.00 & 48.88\% & 0.9 & 0.00 & 9.78\% \\
G34 & 0.5 & -9.20 & 48.88\% & 0.5 & 0.00 & 48.88\% & 0.8 & 0.00 & 19.55\% \\
G35 & 0.5 & -267.20 & 48.88\% & 0.5 & 0.00 & 48.88\% & -- & -- & -- \\
G36 & 0.5 & -268.20 & 48.88\% & 0.5 & 0.00 & 48.88\% & 0.9 & 0.00 & 9.78\% \\
G37 & 0.5 & -284.40 & 48.88\% & 0.5 & 0.00 & 48.88\% & 0.9 & -0.20 & 9.78\% \\
G38 & 0.5 & -289.80 & 48.88\% & 0.5 & 0.00 & 48.88\% & -- & -- & -- \\
G39 & 0.5 & -297.20 & 48.88\% & 0.9 & 0.00 & 9.78\% & 0.9 & 0.00 & 9.78\% \\
G40 & 0.5 & -368.00 & 48.88\% & 0.5 & 0.00 & 48.88\% & 0.6 & 0.00 & 39.10\% \\
G41 & 0.5 & -303.20 & 48.88\% & 0.5 & 0.00 & 48.88\% & -- & -- & -- \\
G42 & 0.5 & -341.80 & 48.88\% & 0.9 & 0.00 & 9.78\% & 0.8 & 0.00 & 19.55\% \\
G43 & 0.5 & -111.80 & 48.88\% & 0.9 & 0.00 & 9.78\% & 0.9 & 0.00 & 9.78\% \\
G44 & 0.5 & -77.60 & 48.88\% & 0.9 & 0.00 & 9.78\% & 0.9 & 0.00 & 9.78\% \\
G45 & 0.5 & -143.80 & 48.88\% & 0.5 & 0.00 & 48.88\% & 0.8 & 0.00 & 19.55\% \\
G46 & 0.5 & -95.20 & 48.88\% & 0.5 & 0.00 & 48.88\% & 0.9 & 0.00 & 9.78\% \\
G47 & 0.5 & -144.00 & 48.88\% & 0.5 & 0.00 & 48.88\% & 0.6 & 0.00 & 39.10\% \\
G48 & 0.6 & -155.20 & 39.10\% & 0.7 & -12.00 & 29.33\% & 0.8 & 0.00 & 19.55\% \\
G49 & 0.6 & -111.20 & 39.10\% & 0.7 & -0.80 & 29.33\% & 0.8 & 0.00 & 19.55\% \\
G50 & 0.6 & -92.80 & 39.10\% & 0.7 & -2.80 & 29.33\% & 0.9 & 0.00 & 9.78\% \\
G51 & 0.5 & -126.40 & 48.88\% & 0.5 & 0.00 & 48.88\% & 0.7 & 0.00 & 29.33\% \\
G52 & 0.5 & -110.40 & 48.88\% & 0.5 & 0.00 & 48.88\% & 0.9 & 0.00 & 9.78\% \\
G53 & 0.5 & -108.00 & 48.88\% & 0.5 & 0.00 & 48.88\% & 0.9 & 0.00 & 9.78\% \\
G54 & 0.5 & -90.80 & 48.88\% & 0.5 & 0.00 & 48.88\% & -- & -- & -- \\
\midrule
Coverage & 54/54 & -- & 47.97\% & 52/54 & -- & 36.84\% & 40/54 & -- & 18.82\% \\
\end{longtable}

\end{landscape}

\end{document}